\documentclass[11pt]{amsart}

\usepackage[a4paper,margin=1in]{geometry}
\usepackage{amsmath,amssymb,amsthm,mathtools}
\usepackage{enumitem}
\usepackage{float}
\usepackage{tikz}
\usetikzlibrary{calc,positioning,arrows.meta,decorations.pathreplacing,fit,backgrounds}
\usepackage{hyperref}

\hypersetup{
  colorlinks=true,
  linkcolor=blue,
  citecolor=blue,
  urlcolor=blue,
  pdftitle={K-spherical horospherical averages on the Nagao quotient},
  pdfauthor={Sanghoon Kwon}
}

\theoremstyle{definition}
\newtheorem{definition}{Definition}[section]
\newtheorem{remark}[definition]{Remark}

\newtheorem{convention}[definition]{Convention}

\theoremstyle{plain}
\newtheorem{lemma}[definition]{Lemma}
\newtheorem{proposition}[definition]{Proposition}
\newtheorem{corollary}[definition]{Corollary}
\newtheorem{theorem}[definition]{Theorem}

\newcommand{\Fq}{\mathbb F_q}
\newcommand{\Floc}{\mathbb F_q(\!(t^{-1})\!)}
\newcommand{\cO}{\mathcal O}
\newcommand{\bG}{\overline G}
\newcommand{\GammaP}{\Gamma_{\mathrm{PGL}}}
\newcommand{\bK}{\overline K}
\newcommand{\Tree}{\mathcal T}
\newcommand{\ZZ}{\mathbb Z}
\newcommand{\CC}{\mathbb C}
\newcommand{\TV}{\operatorname{TV}}
\newcommand{\fwd}{\mathrm{fwd}}
\newcommand{\bwd}{\mathrm{bwd}}
\newcommand{\ev}{\mathrm{ev}}
\newcommand{\Desc}{\operatorname{Desc}}
\newcommand{\ang}{\langle}
\newcommand{\rang}{\rangle}

\definecolor{figblue}{RGB}{43,83,115}
\definecolor{figteal}{RGB}{42,112,105}
\definecolor{figorange}{RGB}{166,91,48}
\definecolor{figgray}{RGB}{112,119,124}

\tikzset{
  every picture/.style={line cap=round,line join=round},
  vertex/.style={circle,draw=figblue,fill=figblue,inner sep=1.5pt},
  terminal/.style={circle,draw=figteal,fill=figteal,inner sep=1.35pt},
  openvertex/.style={circle,draw=figblue,fill=white,line width=.58pt,
    inner sep=1.95pt},
  axis/.style={line width=.92pt,draw=figblue},
  accent/.style={line width=1.08pt,draw=figblue},
  incoming/.style={->,line width=1.05pt,draw=figorange},
  shadowedge/.style={line width=.44pt,draw=figblue!52},
  horoline/.style={densely dashed,line width=.48pt,draw=figgray!75},
  guide/.style={densely dotted,line width=.36pt,draw=figgray!42},
  lab/.style={font=\footnotesize,text=black!80,inner sep=1.1pt},
  smalllab/.style={font=\scriptsize,text=black!78,inner sep=.8pt},
  graylab/.style={font=\footnotesize,text=figgray,inner sep=1pt},
  smallgraylab/.style={font=\scriptsize,text=figgray,inner sep=.7pt},
  figuretitle/.style={font=\footnotesize\bfseries,text=figblue,
    fill=white,rounded corners=2pt,inner xsep=3pt,inner ysep=1.8pt},
  callout/.style={font=\footnotesize,text=black!80,align=center,
    fill=white,rounded corners=2pt,inner xsep=3pt,inner ysep=2pt},
  subtree/.style={line width=.44pt,draw=figblue!48},
  terminalband/.style={rounded corners=5pt,draw=figteal!72,
    fill=figteal!5,line width=.60pt,densely dashed},
  panel/.style={rounded corners=6pt,draw=figgray!32,fill=figgray!2,
    line width=.44pt}
}

\title[Spherical averages on the Nagao quotient]
{$K$-spherical horospherical averages on the Nagao quotient:\ tree combinatorics and exact discrepancy}
\author{Sanghoon Kwon}
\address{Department of Mathematics Education, Catholic Kwandong University,
Gangneung 25601, Republic of Korea}
\email{skwon@cku.ac.kr}
\date{September 15, 2026}
\thanks{
This work was supported by the Basic Science Research Program through the
National Research Foundation of Korea (NRF), funded by the Ministry of
Education (No.~RS-2025-25415913).}



\begin{document}

\begin{abstract}
Let \(q\) be a prime power, \(F=\mathbb F_q(\!(t^{-1})\!)\),
\(G=\mathrm{SL}_2(F)\),
\(\Gamma=\mathrm{SL}_2(\mathbb F_q[t])\), and
\(K=\mathrm{SL}_2(\mathbb F_q[\![t^{-1}]\!])\), and let \(U<G\) be the upper
unipotent subgroup.  We study right \(K\)-spherical averages along \(U\) on
\(X=\Gamma\backslash G\).  Expanding translates of compact \(U\)-orbits and
compact-open F{\o}lner-ball averages become terminal layers of rooted
descendant shadows in the Bruhat--Tits tree.  In the even sector, we compute
the Haar height law and
signed finite-scale discrepancy exactly.  This yields \(K\)-spherical
equidistribution for compact-orbit translates
and, for irrational boundary endpoints, for F{\o}lner-ball averages.  For a
depth-\(N\) shadow rooted at height \(k\), with cutoff \(M=N-k\ge0\),
bounded-profile errors are \(O_q(q^{-M})\) in the backward state and
\(O_q(q^{-M/2})\) uniformly for moving roots, while the
shadow law eventually agrees exactly with the Haar law on every fixed finite
height window.  For \(|\Phi(2m)|\le Cq^{\alpha m}\), \(\alpha<2\), three rate
regimes arise, with a linear-in-scale factor at \(\alpha=1\) and explicit
moving-root dependence.  Artin continued-fraction digits eventually encode the
cutoff and these rates excursion by excursion through individual digit
degrees.
\end{abstract}

\maketitle
\tableofcontents

\section{Introduction}

\subsection{From horospherical averages to tree shadows}
Unipotent and horospherical dynamics on homogeneous spaces are commonly studied
through orbit closures, invariant measures, non-escape of mass, and
equidistribution.  The results of Dani, Ratner, Margulis--Tomanov, and Mohammadi
form the general homogeneous-dynamical background
\cite{Dani1981,Dani1984,Dani1986,Ratner1991Duke,Ratner1991Annals,
MargulisTomanov1994,Mohammadi2011}.  For groups acting on trees,
Ciobotaru--Finkelshtein--Sert proved measure-rigidity results and
equidistribution statements for large compact horospherical orbits
\cite[Proposition~1.6 and Corollary~1.7]{CFS2021}.  They subsequently studied
non-escape of mass and equidistribution for dense horospherical orbits along
good F{\o}lner sequences, using the quotient-graph Markov chain and moving
starting points \cite{CFS2022}.  More specifically, the oriented transition
probabilities on a Nagao ray appear in
\cite[Example~3.1]{CFS2022}, the shadow-to-chain correspondence in
\cite[Lemma~3.9]{CFS2022}, and total-variation convergence with moving initial
states in \cite[Lemma~3.6]{CFS2022}.  The same work establishes an
\(o(r^{-n})\) total-variation rate for projected spheres on geometrically
finite quotients, for some \(r>1\) that can be made effective
\cite[Theorem~E(3)]{CFS2022}, and relates projected-horospherical rates to a
geometric Diophantine exponent \cite[Remark~1.5]{CFS2022}.  These results
provide the most directly relevant general framework for the two families
considered here.  In a complementary direction, Kwon established exponential
mixing for geodesic translation and effective weighted orbit counting on
Bruhat--Tits trees under the non-arithmeticity, fullness, and weighted
spectral-gap hypotheses of that work \cite{Kwon2018}.  That setting concerns
geodesic correlations and counting rather than the rooted descendant shadows
computed below.

There are two convention differences worth making explicit.  The cited tree
papers use the left quotient \(G/\Gamma\) and the full horospherical group,
whereas we use the right quotient \(\Gamma\backslash G\) and its unipotent
radical \(U\).  Inversion identifies the two quotient conventions.  On
right \(K\)-invariant observables the additional compact diagonal factor in
the full horospherical group is invisible; this comparison is detailed in
Section~\ref{sec:spherical-dictionary}.  We nevertheless prove the orbit
dichotomy needed here from the local-field homogeneous theorem in
\cite[Corollary~4.2]{Mohammadi2011}, rather than inferring a statement for
\(U\) directly from a theorem formulated for the larger tree group.

Here we focus on a specialized quantitative problem.  Let \(q\) be a prime power,
set \(F=\Floc\) and \(\cO=\Fq[\![t^{-1}]\!]\), and put
\[
G=\mathrm{SL}_2(F),\qquad
\Gamma=\mathrm{SL}_2(\Fq[t]),\qquad
K=\mathrm{SL}_2(\cO),\qquad
X=\Gamma\backslash G.
\]
We write \(\mu_X\) for the normalized \(G\)-invariant probability measure on
\(X\) and restrict attention to right \(K\)-invariant observables.  Under this
spherical restriction, the quotient \(X/K\) is encoded by the quotient of the
Bruhat--Tits tree \(\Tree\), and the Nagao quotient graph is a ray.  Consequently,
a \(K\)-spherical observable is a function of one height variable.  The averaging
problem then becomes a problem about height distributions in finite rooted
shadows of \(\Tree\).

Let \(h_K:X\to2\ZZ_{\ge0}\) be the quotient-height map
defined in Section~\ref{sec:spherical-dictionary}.  A \emph{height shell} is a
fiber \(h_K^{-1}(j)\).  We say that probability measures \(\lambda_N\) converge
\emph{\(K\)-spherically} to \(\mu_X\) when
\((h_K)_*\lambda_N\to(h_K)_*\mu_X\) in total variation.  Equivalently, their
averages converge uniformly over right \(K\)-invariant test functions of
supremum norm at most one.  This is the precise sense of equidistribution used
throughout; no convergence against general, non-spherical observables is
claimed.

The Nagao ray goes back to Nagao's theorem and Serre's tree interpretation of
arithmetic groups over function fields; useful tree-lattice and
fundamental-domain references include
\cite{Nagao1959,SerreTrees,Tits1979,Lubotzky1994,BassLubotzky2001,Kwon2020}.
Function-field continued fractions code modular-ray geodesics; we use the
normalizations of \cite{Paulin2002,BroisePaulin2007}; see also
\cite{PaulinShapira2020} for related arithmetic applications.  Related counting and
equidistribution results on trees and over local function fields include
\cite{BroiseParkkonenPaulin2019,KwonLim2021,Bravo2024,ParkkonenPaulin2024}.

The general measure-rigidity and non-spherical equidistribution theorems serve
as background.  Our analysis specializes to the \(\mathrm{SL}_2\) even sector
and derives explicit finite-scale formulas for the resulting \(K\)-spherical
Nagao chain.  More precisely, we obtain:
\begin{enumerate}[label=(\roman*),leftmargin=2.2em]
\item an explicit Haar height law in the even \(\mathrm{SL}_2\)-sector;
\item exact shell formulas, including a top-shell-minus-tail discrepancy
identity and a first-turn decomposition;
\item an exact single-shadow formula for every expanding compact
\(U\)-orbit translate;
\item eventual equality, rather than only convergence, for compactly
supported spherical tests; and
\item scale-by-scale bounds for standard F{\o}lner-ball averages, including
controlled unbounded profiles, the transition at growth exponent \(\alpha=1\),
and excursion-by-excursion dependence on individual continued-fraction digit
degrees.
\end{enumerate}
The additional information available in this specialized setting consists of
exact spherical identities and their arithmetic finite-scale consequences.
The formulas identify the error at each scale, give eventual equality on every
fixed finite height window, and express the orbit-dependent cutoff through
individual continued-fraction excursions.  The required orbit dichotomy is
supplied by Mohammadi's homogeneous orbit-closure theorem cited above.  In the
estimates below, the powers of \(q\) and the dependence on the moving root and
continued-fraction data are displayed explicitly; the permitted dependence of
each multiplicative implicit constant is recorded with the corresponding
estimate.

For a depth-\(N\) descendant shadow rooted at quotient height \(k\), define its
cutoff by \(M=N-k\).  Whenever \(M\ge0\), the exact signed-discrepancy formulas
imply a total-variation bound \(O_q(q^{-M/2})\), uniformly in the root height
and orientation, with the sharper bound \(O_q(q^{-M})\) in the backward case.
The comparison with \cite{CFS2022} is one of scope and resolution.  That work
treats projected full spheres on general geometrically finite tree quotients
and gives a rate through an effective \(r>1\), together with a formulation in
terms of a geometric Diophantine exponent.  The present Nagao
\(K\)-spherical model instead concerns rooted descendant terminal layers and
retains exact signed discrepancies, explicit powers of \(q\), and individual
continued-fraction digit data.  Since the averaging objects and levels of
generality differ, the two quantitative statements are complementary.

\subsection{Two dynamical families}
We treat two families in parallel.  First, let \(Y=\Gamma gU\subset X\) be a
compact right \(U\)-orbit, where
\[
U=\left\{u(s)=\begin{pmatrix}1&s\\0&1\end{pmatrix}:s\in F\right\}.
\]
For the diagonal element
\[
a_n=\begin{pmatrix}t^n&0\\0&t^{-n}\end{pmatrix}\in G,
\]
the expanding direction on the right quotient \(\Gamma\backslash G\) is
\(a_{-n}\).  For \(n\ge0\), we therefore set
\[
\lambda_{Y,n}^{\mathrm{per}}(f)
  =\int_Y f(y a_{-n})\,d\nu_Y(y),
\]
where \(\nu_Y\) is the normalized Haar measure on \(Y\); the superscript
\(\mathrm{per}\) stands for periodic, equivalently compact-orbit.  At the level
of \(K\)-spherical pushforwards, every compact \(U\)-orbit is represented by a
unique standard orbit \(Y_m=\Gamma a_mU\).  If \(n>m\), the translated measure
has the exact spherical law
\[
(h_K)_*\lambda_{Y,n}^{\mathrm{per}}
   =\mu^{\fwd}_{1,\,2(n-m)-1}.
\]
The sign is essential: with this right-action convention, translation by
\(a_{-n}\) is expanding, whereas \(Ya_n\) contracts the standard closed
horosphere in spherical projection.

Here and below \(\mu^{\fwd}_{k,N}\) and \(\mu^{\bwd}_{k,N}\) denote two
normalized height laws.  They come from depth-\(N\) non-backtracking descendant
shadows rooted at height \(k\), with the incoming edge pointing respectively
toward or away from the cusp; the formal definition is given in
Section~\ref{sec:rooted-tree}.

Second, let \(x=\Gamma g\in X\) have dense right \(U\)-orbit.  Since \(xU\) has
no finite invariant measure, normalize additive Haar measure on \(F\) by
\(ds(\cO)=1\).  For \(N\ge0\), put
\[
U_N=\{u(s):s\in t^N\cO\},
\]
and define the standard F{\o}lner-ball averages
\[
\lambda_{x,N}^{\mathrm{den}}(f)
  =q^{-N}\int_{t^N\cO}f(xu(s))\,ds.
\]
Here \(\mathrm{den}\) stands for dense-orbit.
The ball \(t^N\cO\) is the union of \(q^N\) cosets modulo \(\cO\), and those
cosets parametrize exactly one terminal descendant layer.  Identify
\(\partial\Tree\) with \(\mathbb P^1(F)\), with \(\infty=[1:0]\), and write
\(\xi_g=g\infty\) for the endpoint fixed by \(gUg^{-1}\).  Viewing \(gK\) as
an even vertex of \(\Tree\), the root is the \(N\)-th vertex of the geodesic
from \(gK\) to \(\xi_g\).
Its quotient height \(k_N(x)\) moves with \(N\).  The orientation
\(\varepsilon_N(x)\in\{\fwd,\bwd\}\) records the type of the incoming rooted
edge: \(\fwd\) means that it points toward the cusp, whereas \(\bwd\) means
that it points away from the cusp.  The relevant cutoff is
\[
M_N(x)=N-k_N(x).
\]
For fixed \(x\), we abbreviate
\(k_N=k_N(x)\), \(\varepsilon_N=\varepsilon_N(x)\), and
\(M_N=M_N(x)\).
The cutoff measures how far the depth-\(N\) shadow has progressed past its root
height toward the bottom of the quotient ray.  For an \emph{irrational}
endpoint, meaning \(\xi_g\notin\mathbb P^1(\Fq(t))\), one has
\(M_N\to\infty\); rationality is independent of the chosen lift \(g\).
The constant and increasing blocks of \(M_N\) correspond, respectively, to
the ascending and descending halves of the continued-fraction cusp excursions.

We next record the fixed-root, moving-root, and excursion-wise estimates.  Write
\(\rho^{\ev}\) for the even limiting law displayed in \eqref{eq:rhoev} below,
and put
\[
\ang\Phi\rang_{\rho^{\ev}}
  =\sum_{m\ge0}\Phi(2m)\rho^{\ev}(2m).
\]
First fix a forward root height \(k\).  For a height profile \(\Phi\), suppose
that \(|\Phi(2m)|\le Cq^{\alpha m}\) for some \(C>0\) and \(\alpha<2\), and put
\(\beta=(2-\alpha)/2\).  Then, for \(N\ge k\) with \(k+N\) even,
Proposition~\ref{prop:effective-forward} gives
\begin{equation}\label{eq:intro-fixed-forward-rates}
\left|\int\Phi\,d\mu^{\fwd}_{k,N}
      -\ang\Phi\rang_{\rho^{\ev}}\right|
\ll_{q,\alpha,k}
\begin{cases}
Cq^{-N/2},&\alpha<1,\\
CNq^{-N/2},&\alpha=1,\\
Cq^{-\beta N},&1<\alpha<2.
\end{cases}
\end{equation}
The linear factor at \(\alpha=1\) is the borderline geometric-series loss in
the first-turn decomposition.  For moving roots and bounded profiles,
Proposition~\ref{prop:moving-TV} removes the dependence on a fixed initial
height and, whenever \(M_N\ge0\), gives the uniform error
\(O_q(q^{-M_N/2})\), with the sharper \(O_q(q^{-M_N})\) in a backward state.

The moving-root estimate extends to controlled unbounded profiles without
hiding the dependence on \(k_N\).  Put
\[
E_N(\Phi;x)=
\left|\lambda_{x,N}^{\mathrm{den}}(\Phi\circ h_K)
      -\ang\Phi\rang_{\rho^{\ev}}\right|.
\]
Under the same growth hypothesis, Theorem~\ref{thm:dense-growth} gives,
whenever \(M_N\ge0\),
\begin{equation}\label{eq:intro-moving-growth-rates}
E_N(\Phi;x)\ll
\begin{cases}
Cq^{-M_N/2},&\alpha<1,\\
C(M_N+k_N+1)q^{-M_N/2},&\alpha=1,\\
Cq^{-\beta M_N+(\alpha-1)k_N},&1<\alpha<2.
\end{cases}
\end{equation}
Here the implicit constant depends only on \(q\) when \(\alpha=1\), and only on
\(q,\alpha\) otherwise; in particular, no dependence on \(x,N\), or \(k_N\) is
hidden.  Since \(M_N+k_N=N\), the critical loss is \(N+1\), equivalently
logarithmic in the ball volume \(q^N\).  Thus convergence at \(\alpha=1\)
follows whenever \((N+1)q^{-M_N/2}\to0\).  In the last regime, a convenient
sufficient condition for convergence is
\[
k_N\le\vartheta M_N\quad\hbox{eventually},
\qquad
\vartheta<\frac{2-\alpha}{2(\alpha-1)}.
\]

Corollary~\ref{cor:cf-rates} expresses the same information on individual
continued-fraction excursions.  If \(d_r\) is the degree of the \(r\)-th Artin
digit and \(B_r\) the corresponding return scale, then, for
\(0\le j\le d_r-1\),
\begin{equation}\label{eq:intro-cf-rates}
E_{B_r+j}(\Phi;x)\ll_{q,\alpha}Cq^{-\beta B_r}
\end{equation}
on the ascending half.  On the descending half, where
\(N=B_r+d_r+j\), one has
\begin{equation}\label{eq:intro-cf-descending-rate}
E_{B_r+d_r+j}(\Phi;x)\ll
\begin{cases}
Cq^{-B_r/2-j},&\alpha<1,\\
C(B_r+d_r+j+1)q^{-B_r/2-j},&\alpha=1,\\
Cq^{-\beta B_r+(\alpha-1)d_r-j},&1<\alpha<2.
\end{cases}
\end{equation}
The implicit constant has the same dependence as in
\eqref{eq:intro-moving-growth-rates}.  Hence the specialized Nagao formula
records the degree \(d_r\) of the digit associated with the excursion currently
being traversed.  This excursion-level arithmetic refinement complements the
general formulation through a geometric Diophantine exponent in
\cite{CFS2022}.  The three regimes arise from the same geometric sum in the
first-turn decomposition.  The dense-orbit construction may be summarized by
\[
x=\Gamma g\longmapsto \xi_g
 \longmapsto \bigl(k_N(x),\varepsilon_N(x),M_N(x)\bigr)
 \longmapsto \mu^{\varepsilon_N(x)}_{k_N(x),N}.
\]

Although dynamically different, the two families share the same spherical
combinatorial engine:
\[
\begin{gathered}
\text{horospherical average on }\Gamma\backslash G
\longrightarrow
\text{spherical pushforward on }\Gamma\backslash G/K\\
\longrightarrow
\text{height law on the Nagao ray}.
\end{gathered}
\]

\subsection{The rooted limiting law}
The group \(\mathrm{SL}_2(F)\) preserves the bipartition of the Bruhat--Tits tree.
After fixing the base vertex, we work in the even sector.  The limiting height
law on \(2\ZZ_{\ge0}\) is
\begin{equation}\label{eq:rhoev}
\rho^{\ev}(0)=\frac{q-1}{q},\qquad
\rho^{\ev}(2m)=(q^2-1)q^{-2m-1}\quad (m\ge 1).
\end{equation}
This distribution has two interpretations.  It is the limit of the rooted shadow
distributions, and it is also the \(K\)-spherical pushforward of the
\(G\)-invariant probability measure on the Nagao quotient.

\subsection{Main results}
The first result is purely combinatorial.  The phrase \emph{admissible parity}
means that all terminal heights under consideration lie in \(2\ZZ_{\ge0}\); for a
root at height \(k\) and depth \(N\), this is the condition \(k+N\equiv 0\pmod2\).

\begin{theorem}[Exact rooted shadow discrepancy]\label{thm:rooted-shadow}
Let \(\mu^{\bwd}_{k,N}\) and \(\mu^{\fwd}_{k,N}\) be the normalized height
distributions of the backward and forward depth-\(N\) descendant shadows whose
current root has height \(k\), where \(k\ge0\) in the backward case and
\(k\ge1\) in the forward case.  Assume admissible parity.  For
\(\Phi\in L^1(\rho^{\ev})\), put
\[
\ang\Phi\rang_{\rho^{\ev}}=\sum_{m\ge0}\Phi(2m)\rho^{\ev}(2m).
\]
For a backward initial state with \(M=N-k\ge0\), one has the exact identity
\begin{equation}\label{eq:intro-back-error}
\int \Phi\,d\mu^{\bwd}_{k,N}-\ang\Phi\rang_{\rho^{\ev}}
 = q^{-M-1}\Phi(M)-\sum_{2m\ge M+2}\rho^{\ev}(2m)\Phi(2m).
\end{equation}
For a forward initial state,
\begin{align}\label{eq:intro-forward-error}
\int \Phi\,d\mu^{\fwd}_{k,N}-\ang\Phi\rang_{\rho^{\ev}}
&=q^{-N}\bigl(\Phi(k+N)-\ang\Phi\rang_{\rho^{\ev}}\bigr)  \\
&\quad +(q-1)\sum_{r=0}^{N-1}q^{-r-1}
   \left(\int \Phi\,d\mu^{\bwd}_{k+r-1,N-r-1}
          -\ang\Phi\rang_{\rho^{\ev}}\right). \nonumber
\end{align}
Consequently, if \(\Phi\in L^1(\rho^{\ev})\), then, for each fixed initial
height, the backward and forward shadow averages converge to
\(\ang\Phi\rang_{\rho^{\ev}}\) as \(N\to\infty\) through admissible parity.
If \(|\Phi(2m)|\le Cq^{\alpha m}\) with \(\alpha<2\), then the convergence is
exponential, with explicit exponents given in
Propositions~\ref{prop:effective-back} and \ref{prop:effective-forward}.
\end{theorem}

The two exact identities in Theorem~\ref{thm:rooted-shadow} are proved in
Propositions~\ref{prop:back-discrepancy} and
\ref{prop:forward-discrepancy}; convergence and rates are then supplied by
Theorem~\ref{thm:rooted-equidistribution} and
Propositions~\ref{prop:effective-back}--\ref{prop:effective-forward}.

The next theorem is the interface between the tree calculation and homogeneous
dynamics.  We write \(\mu_X\) for the normalized \(G\)-invariant probability
measure on \(X\), and use the spherical height map
\(h_K:X\to2\ZZ_{\ge0}\) introduced above.

\begin{theorem}[Spherical shadow principle]
\label{thm:shadow-principle}
Let \(\lambda_T\) be a sequence of probability measures on \(X\), indexed by a
scale parameter \(T\).  Assume that, for each \(T\), its height pushforward
admits a finite rooted shadow decomposition
\[
(h_K)_*\lambda_T
 =\sum_{\ell=1}^{L_T}c_{\ell,T}
   \mu^{\varepsilon_{\ell,T}}_{k_{\ell,T},D_{\ell,T}},
\qquad
c_{\ell,T}\ge0,
\qquad
\sum_{\ell=1}^{L_T}c_{\ell,T}=1,
\]
where each component has admissible parity and
\(\varepsilon_{\ell,T}\in\{\fwd,\bwd\}\), with
\(k_{\ell,T},D_{\ell,T}\ge0\) and \(k_{\ell,T}\ge1\) whenever
\(\varepsilon_{\ell,T}=\fwd\).  Define
\[
M_T=\min_{1\le\ell\le L_T}\bigl(D_{\ell,T}-k_{\ell,T}\bigr),
\qquad\text{and assume that }M_T\ge0.
\]
Then for every bounded right \(K\)-invariant observable
\(f(\Gamma g)=\Phi(h_K(\Gamma g))\),
\[
\left|\int_X f\,d\lambda_T-\int_X f\,d\mu_X\right|
\ll_q \|\Phi\|_\infty q^{-M_T/2}.
\]
If all shadow components are backward components, the sharper bound
\(O_q(\|\Phi\|_\infty q^{-M_T})\) holds.  In particular, if \(M_T\to\infty\), then
\(\lambda_T\) equidistributes toward \(\mu_X\) after testing against bounded right
\(K\)-invariant functions.
\end{theorem}

\begin{proof}
For a right \(K\)-invariant observable \(f=\Phi\circ h_K\),
\[
\int_X f\,d\lambda_T-\int_X f\,d\mu_X
=\int \Phi\,d(h_K)_*\lambda_T-\int \Phi\,d\rho^{\ev}.
\]
Using the finite shadow decomposition and the triangle inequality, the absolute
value is at most
\[
\sum_{\ell=1}^{L_T}c_{\ell,T}\,\|\Phi\|_\infty\,
\bigl\|\mu^{\varepsilon_{\ell,T}}_{k_{\ell,T},D_{\ell,T}}-\rho^{\ev}\bigr\|_{\TV}.
\]
The uniform moving-root estimate, Proposition~\ref{prop:moving-TV}, gives
\[
\bigl\|\mu^{\varepsilon_{\ell,T}}_{k_{\ell,T},D_{\ell,T}}-\rho^{\ev}\bigr\|_{\TV}
\ll_q q^{-(D_{\ell,T}-k_{\ell,T})/2}\le q^{-M_T/2}.
\]
Averaging over \(\ell\) gives the stated bound.  If every component is backward,
the sharper estimate in Proposition~\ref{prop:moving-TV} gives
\(O_q(\|\Phi\|_\infty q^{-M_T})\).  Finally, by
Proposition~\ref{prop:haar-rho}, \(\int_X f\,d\mu_X=\int \Phi\,d\rho^{\ev}\).
\end{proof}

Theorem~\ref{thm:shadow-principle} is a reusable interface.  Once a dynamical
average is shown to realize a rooted shadow, equidistribution follows from the
combinatorics.

\begin{theorem}[Compact translates and dense-orbit F{\o}lner balls]
\label{thm:two-families}
The two dynamical families have the following exact spherical realizations.
\begin{enumerate}[label=(\roman*)]
\item For every compact right \(U\)-orbit \(Y\), there is an integer \(m=m(Y)\)
such that, for \(n\in\ZZ_{\ge0}\) with \(n>m\),
\[
(h_K)_*\lambda_{Y,n}^{\mathrm{per}}
 =\mu^{\fwd}_{1,\,2(n-m)-1}.
\]
In particular, its cutoff is \(2(n-m)-2\).
\item Let \(x\in X\) and \(N\in\ZZ_{\ge0}\), and let
\(\lambda_{x,N}^{\mathrm{den}}\) be the standard F{\o}lner-ball average over
\(t^N\cO\).  Then
\[
(h_K)_*\lambda_{x,N}^{\mathrm{den}}
 =\mu^{\varepsilon_N(x)}_{k_N(x),N}.
\]
If, for a lift \(x=\Gamma g\), the boundary point \(\xi_g\) is irrational,
equivalently if \(xU\) is dense by Lemma~\ref{lem:irrational-dense}, then
\(M_N=N-k_N(x)\to\infty\).  Hence
\(\lambda_{x,N}^{\mathrm{den}}\) equidistributes \(K\)-spherically toward
\(\mu_X\).  For every bounded profile \(\Phi\), whenever \(M_N\ge0\),
\[
\left|\lambda_{x,N}^{\mathrm{den}}(\Phi\circ h_K)
      -\ang\Phi\rang_{\rho^{\ev}}\right|
 \ll_q\|\Phi\|_\infty q^{-M_N/2};
\]
in particular, this estimate holds for all sufficiently large \(N\).
\end{enumerate}
\end{theorem}

Part~(i) is Proposition~\ref{prop:compact-shadow}.  Part~(ii) combines
Proposition~\ref{prop:dense-exact-shadow}, Lemma~\ref{lem:irrational-dense},
Corollary~\ref{cor:MN-diverges}, and Theorem~\ref{thm:dense-bounded}.

Both families have a stronger finite-window phenomenon.  If a right
\(K\)-invariant observable has height profile supported in finitely many
shells, then its discrepancy is identically zero once the relevant cutoff has
passed that support.  See Theorem~\ref{thm:compact-support-exact} and
Corollary~\ref{cor:dense-finite-exact}.

\subsection{Organization}
Section~\ref{sec:spherical-dictionary} fixes the homogeneous notation and
proves the spherical Haar law.  Section~\ref{sec:rooted-tree} establishes the
rooted-shadow formulas.  Section~\ref{sec:compact} classifies compact
\(U\)-orbits up to spherical projection and computes their expanding
translates.  Section~\ref{sec:dense} treats standard F{\o}lner-ball averages,
moving roots, and continued-fraction rates.  A brief conclusion records the
scope of the results.

\section*{Acknowledgements}
This work was supported by the Basic Science Research Program through the
National Research Foundation of Korea (NRF), funded by the Ministry of
Education (No.~RS-2025-25415913).

\section{Homogeneous setup and the spherical dictionary}\label{sec:spherical-dictionary}

\subsection{Notation and conventions}
Let \(q\) be a prime power.  Put
\[
\begin{gathered}
A=\Fq[t],\qquad \pi=t^{-1},\qquad
F=\Floc,\qquad \cO=\Fq[\![t^{-1}]\!],\\
G=\mathrm{SL}_2(F),\qquad
\Gamma=\mathrm{SL}_2(\Fq[t])
\end{gathered}
\]
and set
\[
X=\Gamma\backslash G.
\]
The group \(\Gamma\) is the standard nonuniform arithmetic lattice in \(G\);
the Nagao reduction recalled below in particular gives finite covolume.  We
write \(\mu_X\) for the resulting normalized \(G\)-invariant probability
measure.
We normalize the valuation and absolute value by
\[
v_\infty(\pi)=1,\qquad |t|=q.
\]
Thus \(\cO=\{s\in F:|s|\le1\}\), and \(F/A\) has the compact fundamental
domain \(t^{-1}\cO\).  For \(\xi\in F\), its polynomial part
\(\lfloor\xi\rfloor\in A\) is the unique polynomial satisfying
\(\xi-\lfloor\xi\rfloor\in t^{-1}\cO\).

We write points of \(X\) as \(x=\Gamma g\).  All actions below are right actions.
Thus the right \(U\)-orbit of \(x\) is \(xU=\{\Gamma gu:u\in U\}\).  The standard
compact subgroup and the upper unipotent subgroup are
\[
K=\mathrm{SL}_2(\cO),
\qquad
U=\left\{u(s)=\begin{pmatrix}1&s\\0&1\end{pmatrix}:s\in F\right\}.
\]
When an element \(u(s)\) acts on \(\Tree\), the same symbol denotes its
projective image in \(\mathrm{PGL}_2(F)\).
Let \(\Tree\) denote the Bruhat--Tits tree of \(\mathrm{PGL}_2(F)\), whose
lattice model is recalled in the next subsection, and let \(v_0=[\cO^2]\) be
its standard vertex.
For \(r\in F^\times\), set
\[
d(r)=\begin{pmatrix}r&0\\0&r^{-1}\end{pmatrix},
\qquad
M_\infty=\{d(\eta):\eta\in\cO^\times\}.
\]
Identifying \(\partial\Tree\) with \(\mathbb P^1(F)\), the group \(U\) fixes
\(\infty=[1:0]\).  For vertices \(v,w\), define the Busemann cocycle by
\[
\beta_\infty(v,w)
 =\lim_{z\to\infty}\bigl(d_\Tree(v,z)-d_\Tree(w,z)\bigr),
\]
where \(z\) tends to \(\infty\) along any geodesic ray.  Thus
\(\beta_\infty(v,w)=1\) when \(w\) is the neighbor of \(v\) toward
\(\infty\).  The subgroup of \(G\) preserving the horosphere through
\(v_0\)---the elliptic horospherical group attached to this endpoint---is the
Busemann kernel
\[
H_\infty
 =\{h\in\operatorname{Stab}_G(\infty):
       \beta_\infty(v_0,\bar h v_0)=0\}
 =M_\infty U
 =\left\{\begin{pmatrix}\eta&s\\0&\eta^{-1}\end{pmatrix}:
          \eta\in\cO^\times,\ s\in F\right\}.
\]
Here \(\bar h\) denotes the projective image of \(h\).
The compact factor \(M_\infty\subset K\) normalizes \(U\), and multiplication
by a unit preserves every ball \(t^N\cO\).  Consequently, after averaging in
the compact factor, an \(H_\infty\)-average and the corresponding \(U\)-average
have the same value on right \(K\)-invariant functions.  More explicitly, if
\(dm\) is Haar probability on \(M_\infty\), then
\[
\int_{M_\infty}\frac1{q^N}\int_{t^N\cO}f(xm u(s))\,ds\,dm
 =\frac1{q^N}\int_{t^N\cO}f(xu(s))\,ds,
\]
because \(d(\eta)u(s)=u(\eta^2s)d(\eta)\).  This explains why the
present spherical calculation may be formulated using \(U\), whereas the
general tree literature is commonly stated for the full horospherical group.

We use the right quotient \(\Gamma\backslash G\).  Results stated on the left
quotient \(G/\Gamma\) are compared through the inversion homeomorphism
\[
\iota:\Gamma\backslash G\longrightarrow G/\Gamma,
\qquad \iota(\Gamma g)=g^{-1}\Gamma,
\]
which sends a right \(U\)-orbit to a left \(U\)-orbit.  We identify \(U\) with
\((F,+)\) and write \(ds\) for additive Haar measure on \(F\), normalized by
\(ds(\cO)=1\).

For \(n\in\ZZ\), set
\[
a_n=\begin{pmatrix}t^n&0\\0&t^{-n}\end{pmatrix}\in G.
\]
Its projective class is \([\operatorname{diag}(t^{2n},1)]\), and hence it
translates the standard apartment by \(2n\) edges.  In the notation
\(\bar a_N=[\operatorname{diag}(t^N,1)]\) introduced below, the precise
relation is
\[
[a_n]=\bar a_{2n}.
\]
Notice the conjugation formulas
\[
a_nu(s)a_{-n}=u(t^{2n}s),
\qquad
a_{-n}u(s)a_n=u(t^{-2n}s).
\]
Because we work with right actions on \(\Gamma\backslash G\), the expanding
translate of a closed upper-unipotent orbit is by \(a_{-n}\).

For a signed measure \(\nu\) on the discrete height space, we use the total-variation norm
\[
\|\nu\|_{\TV}=\sum_j |\nu(j)|.
\]
With this convention,
\(|\int \Phi\,d\nu|\le \|\Phi\|_\infty\|\nu\|_{\TV}\) for bounded height
profiles \(\Phi\).
We write \(A\ll B\) if \(|A|\le C B\) for a constant \(C\) depending only on the
parameters indicated in the subscript.
We write \(A\asymp B\) when both \(A\ll B\) and \(B\ll A\) hold.

\subsection{The Bruhat--Tits tree and the Nagao ray}
It is useful first to pass to
\[
\bG=\mathrm{PGL}_2(F),
\qquad
\GammaP=\mathrm{PGL}_2(\Fq[t]),
\qquad
\bK=\mathrm{PGL}_2(\cO).
\]
Let \(\Tree\) be the Bruhat--Tits tree of \(\bG\).  Its vertices are homothety
classes of \(\cO\)-lattices in \(F^2\).  Two vertices \([L]\) and \([L']\)
are adjacent if representatives can be chosen so that
\[
\pi L\subsetneq L'\subsetneq L,
\qquad L/L'\simeq\Fq.
\]
In particular, every vertex has \(q+1\) neighbors.  Moreover,
\[
\bG/\bK\simeq V(\Tree).
\]
We denote the graph distance on \(V(\Tree)\) by \(d_\Tree\).
Recall that \(v_0\) is the class of \(\cO^2\), and put
\[
\bar a_n=\begin{bmatrix}t^n&0\\0&1\end{bmatrix}\in\bG,
\qquad
v_n=\bar a_n v_0.
\]
The bipartite type of a vertex is
\[
\operatorname{type}(v)=d_\Tree(v,v_0)\pmod 2.
\]
For nonnegative integers \(n\) and \(N\), the two scale conventions used later
are related by the following dictionary:
\[
\begin{array}{ccl}
a_n\in G&\longmapsto&[a_n]v_0=\bar a_{2n}v_0,
\qquad d_\Tree(v_0,[a_n]v_0)=2n,\\[2mm]
t^N\cO\subset F&\longmapsto&\bar a_Nv_0,
\qquad d_\Tree(v_0,\bar a_Nv_0)=N.
\end{array}
\]
Thus the diagonal index \(n\) records half the tree displacement, whereas the
F{\o}lner-ball index \(N\) is already a tree depth.
The auxiliary projective lattice \(\GammaP\) is used only to describe the
tree quotient.  It need not be the projective image of \(\Gamma\); the latter
is a subgroup of \(\GammaP\).  The underlying graph of
\(\GammaP\backslash\Tree\) is the Nagao ray, represented by
\[
v_0-v_1-v_2-\cdots.
\]
The quotient is an edge-indexed graph: the single drawn edge toward the cusp
and the single drawn edge toward the compact part encode different
multiplicities upstairs.  Proposition~\ref{prop:local-height} records those
multiplicities, which are essential for every shadow count below.
We define
\[
h:V(\Tree)\longrightarrow\ZZ_{\ge0}
\]
by requiring \(h(v)=n\) when the image of \(v\) in \(\GammaP\backslash\Tree\) is
represented by \(v_n\); see
\cite{Nagao1959,SerreTrees,Tits1979,Kwon2020}.

\subsection{Local height structure}
\begin{proposition}[Local height structure]\label{prop:local-height}
Let \(v\in V(\Tree)\).  If \(h(v)=0\), then every neighbor of \(v\) has height
\(1\).  If \(h(v)=k\ge1\), then exactly one neighbor of \(v\) has height \(k+1\)
and exactly \(q\) neighbors have height \(k-1\).
\end{proposition}

\begin{proof}
For \(n\ge1\), the stabilizer of \(v_n\) in \(\GammaP\) is
\[
\Gamma_{\mathrm{PGL},n}=\GammaP\cap \bar a_n\bK\bar a_n^{-1}.
\]
It consists of projective classes represented by matrices
\[
\begin{pmatrix}a&b\\0&d\end{pmatrix},
\qquad a,d\in\Fq^\times,
\qquad b\in\Fq[t],\quad \deg b\le n.
\]
The action on the neighbors of \(v_n\) factors through the upper triangular
subgroup of \(\mathrm{PGL}_2(\Fq)\).  This subgroup has two orbits on
\(\mathbb P^1(\Fq)\): the fixed point \(\infty\), corresponding to the unique
neighbor \(v_{n+1}\), and the affine line \(\Fq\), corresponding to the \(q\)
remaining neighbors, all projecting to \(v_{n-1}\).  At \(n=0\), the stabilizer is
\(\mathrm{PGL}_2(\Fq)\), which is transitive on \(\mathbb P^1(\Fq)\); hence all
neighbors project to height \(1\).
\end{proof}

Figure~\ref{fig:local-height} displays the underlying ray and the multiplicities
encoded by its two edge directions.

\begin{figure}[ht]
\centering
\begin{tikzpicture}[x=1cm,y=1cm,>=Stealth]
  \draw[panel] (-.62,-.70) rectangle (6.92,1.10);
  \draw[panel] (.78,-4.34) rectangle (5.22,-.70);

  \foreach \i in {0,...,5}{
    \node[openvertex] (r\i) at (1.1*\i,0) {};
    \node[lab] at (1.1*\i,-.42) {\(\i\)};
  }
  \foreach \i/\j in {0/1,1/2,2/3,3/4,4/5}{\draw[axis] (r\i)--(r\j);}
  \draw[->,axis] (5.6,.25)--(6.3,.25) node[right,lab]{cusp};
  \node[figuretitle] at (2.75,.70)
    {quotient Nagao ray \(0-1-2-\cdots\)};

  \node[vertex,label={[lab]left:height \(k\)}] (v) at (3,-2.25) {};
  \node[vertex,label={[lab]right:unique height \(k+1\)}] (up) at (3,-1.05) {};
  \draw[->,axis] (v)--(up);
  \node[lab,right] at (3.16,-1.62) {toward cusp};
  \foreach \i/\x in {1/-1.2,2/-.4,3/.4,4/1.2}{
    \node[terminal] (down\i) at (3+\x,-3.55) {};
    \draw[subtree] (v)--(down\i);
  }
  \draw[decorate,decoration={brace,mirror,amplitude=4pt}]
    (1.60,-3.80)--(4.40,-3.80) node[midway,below=5pt,lab]
    {\(q\) neighbors of height \(k-1\)};
\end{tikzpicture}
\caption{Local height structure on the Nagao ray.  Above height zero there is one continuation toward the cusp and \(q\) continuations back toward the compact part.}
\label{fig:local-height}
\end{figure}
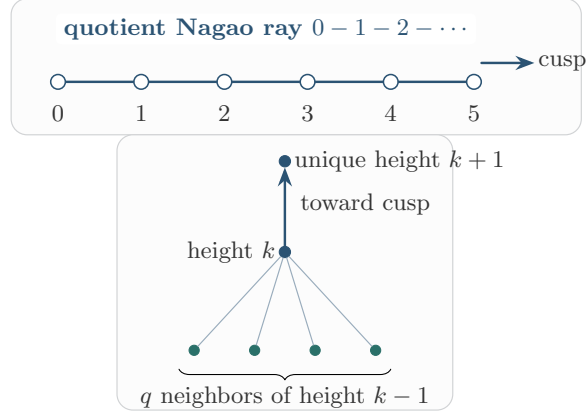

\subsection{The even sector for
\texorpdfstring{\(\mathrm{SL}_2\)}{SL2}}
The group \(\mathrm{SL}_2(F)\) preserves the bipartition of \(\Tree\).  After
fixing \(v_0\), the image of \(G/K\) is the even vertex class.  Nagao reduction
in determinant-one form gives the following precise double-coset statement.

\begin{proposition}[The determinant-one Nagao ray]
\label{prop:sl-nagao}
One has the disjoint decomposition
\[
G=\bigsqcup_{m\ge0}\Gamma a_mK.
\]
Under the action on \(\Tree\), the coset \(\Gamma a_mK\) has quotient height
\(2m\).  Consequently,
\[
\Gamma\backslash G/K\simeq 2\ZZ_{\ge0}.
\]
\end{proposition}

\begin{proof}
Apply the projective Nagao decomposition
\(\bG=\bigsqcup_{r\ge0}\GammaP\bar a_r\bK\), as established in
\cite{Nagao1959,SerreTrees,Kwon2020}.  The image of \(g\in G\) preserves vertex
type, so its projective double coset has even height \(r=2m\).  Choose
representatives \(\gamma\in\mathrm{GL}_2(A)\) and
\(k\in\mathrm{GL}_2(\cO)\).  Since
\(\operatorname{diag}(t^{2m},1)=t^ma_m\) projectively, there is
\(\zeta\in F^\times\) such that
\[
g=\zeta\gamma a_mk.
\]
Put \(\Delta=\det\gamma\in\Fq^\times\),
\[
\gamma'=\gamma\operatorname{diag}(\Delta^{-1},1)\in\Gamma,
\qquad
k_1=\operatorname{diag}(\Delta,1)k\in\mathrm{GL}_2(\cO).
\]
The diagonal unit commutes with \(a_m\), so \(g=\zeta\gamma'a_mk_1\).
Taking determinants gives \(\zeta^2\det k_1=1\); hence
\(\zeta\in\cO^\times\) and
\(k'=\zeta k_1\in K\).  Thus \(g=\gamma'a_mk'\), proving existence in the
determinant-one double quotient.  Finally, the projective class of \(a_m\) is
\([\operatorname{diag}(t^{2m},1)]\), of quotient height \(2m\).  Projective
uniqueness therefore proves disjointness.
\end{proof}

The spherical height map is
\[
h_K:X\longrightarrow2\ZZ_{\ge0},
\qquad
h_K(\Gamma g)=h(\bar g v_0),
\]
where \(\bar g\) is the image of \(g\) in \(\bG\).  This is well defined:
left multiplication by \(\Gamma\) becomes multiplication by a subgroup of
\(\GammaP\), and right multiplication by \(K\) fixes \(v_0\).  Thus \(h_K\)
is constant on right \(K\)-cosets and induces the identification
\(X/K\simeq2\ZZ_{\ge0}\) of Proposition~\ref{prop:sl-nagao}.

\begin{definition}[Spherical observables]
A function \(f:X\to\CC\) is \(K\)-spherical if
\[
f(\Gamma gk)=f(\Gamma g)\qquad(g\in G,\ k\in K).
\]
Equivalently, there is a height profile \(\Phi:2\ZZ_{\ge0}\to\CC\) such that
\[
f(\Gamma g)=\Phi(h_K(\Gamma g)).
\]
\end{definition}

\begin{definition}[\(K\)-spherical convergence]
A sequence of probability measures \((\lambda_n)\) on \(X\) converges
\emph{\(K\)-spherically} to \(\mu_X\) if
\[
\bigl\|(h_K)_*\lambda_n-(h_K)_*\mu_X\bigr\|_{\TV}\longrightarrow0.
\]
This definition deliberately concerns only the right \(K\)-invariant factor of
the measures.
\end{definition}

\subsection{Spherical Haar pushforward}
The normalized \(G\)-invariant probability measure \(\mu_X\) introduced above
has the \(K\)-spherical pushforward
\[
\rho^{\ev}=(h_K)_*\mu_X
\]
on \(2\ZZ_{\ge0}\).

\begin{proposition}[Explicit spherical Haar law]\label{prop:haar-rho}
The measure \(\rho^{\ev}\) is given by \eqref{eq:rhoev}:
\[
\rho^{\ev}(0)=\frac{q-1}{q},\qquad
\rho^{\ev}(2m)=(q^2-1)q^{-2m-1}\qquad(m\ge1).
\]
\end{proposition}

\begin{proof}
By Proposition~\ref{prop:sl-nagao}, the shell at height \(2m\) is represented
by \(\Gamma a_mK\).  If Haar measure on \(G\) is normalized by
\(\operatorname{vol}(K)=1\), unfolding the compact-open double coset gives
\[
\operatorname{vol}(\Gamma\backslash\Gamma a_mK)
   =\frac{1}{|\Gamma\cap a_mKa_m^{-1}|}.
\]
Thus the shell mass is proportional to the inverse order of
\[
\Gamma_m:=\Gamma\cap a_mKa_m^{-1}.
\]
At \(m=0\), one has \(\Gamma_0=\mathrm{SL}_2(\Fq)\), of order
\(q(q^2-1)\).  If \(m\ge1\), direct intersection gives
\[
\Gamma_m=
\left\{
\begin{pmatrix}\alpha&b\\0&\alpha^{-1}\end{pmatrix}:
\alpha\in\Fq^\times,\ b\in A,\ \deg b\le2m
\right\},
\]
and hence \(|\Gamma_m|=(q-1)q^{2m+1}\).  Therefore the ratio of the shell mass
at height \(2m\) to the bottom-shell mass is
\[
\frac{|\Gamma_0|}{|\Gamma_m|}
=\frac{q(q^2-1)}{(q-1)q^{2m+1}}
=(q+1)q^{-2m}.
\]
Hence \(\rho^{\ev}(2m)=\rho^{\ev}(0)(q+1)q^{-2m}\).  Normalization gives
\[
1=\rho^{\ev}(0)\left(1+(q+1)\sum_{m\ge1}q^{-2m}\right)
=\rho^{\ev}(0)\frac{q}{q-1},
\]
so \(\rho^{\ev}(0)=(q-1)/q\).  The displayed formula for \(m\ge1\) follows.
\end{proof}

\begin{remark}
Proposition~\ref{prop:haar-rho} is the point at which the tree calculation is
connected to the homogeneous measure.  Once an averaging measure has height
pushforward close to \(\rho^{\ev}\), it is equidistributed against all
\(K\)-spherical observables.
\end{remark}

\subsection{Shadow realization of spherical averages}
Let \(\lambda\) be a probability measure on \(X\).  For a \(K\)-spherical
observable \(f=\Phi\circ h_K\),
\[
\int_X f\,d\lambda=
\sum_{j\in2\ZZ_{\ge0}}\Phi(j)\,(h_K)_*\lambda(j).
\]
Thus all \(K\)-spherical information in \(\lambda\) is contained in the height
pushforward \((h_K)_*\lambda\).  The central task in the two dynamical
applications below is to show that this pushforward is one of the rooted
descendant distributions described in the next section.

\section{Rooted tree combinatorics}\label{sec:rooted-tree}

All shadows in this section live upstairs in the \((q+1)\)-regular tree, not
in the underlying unweighted quotient ray.  Distinct terminal vertices may
project to the same height and are counted with their full multiplicity.  The
resulting normalized count can be viewed as a finite-time height law, but no
independent random-walk assumption is being introduced: the probabilities are
uniform counting measure on one deterministic descendant layer.

\subsection{Oriented states and descendant shadows}
Let \(e=(z_{-1},z_0)\) be an oriented edge in \(\Tree\), and write \(h(z_0)=k\).
The edge is of forward type if \(h(z_{-1})=k-1\), and of backward type if
\(h(z_{-1})=k+1\).  We write these states as \((\fwd,k)\) and \((\bwd,k)\),
respectively; see Figure~\ref{fig:states}.

\begin{convention}[The bottom backward state]\label{conv:bottom-backward}
The state \((\bwd,0)\) means an oriented edge \(z_{-1}\to z_0\) with
\(h(z_{-1})=1\) and \(h(z_0)=0\).  Since height zero has \(q+1\) neighbors, the
non-backtracking continuations from \((\bwd,0)\) are the \(q\) neighbors of
height \(1\) other than \(z_{-1}\).  Each such continuation enters the forward
state \((\fwd,1)\).  This convention is needed in the first-turn decomposition
when the first turn lands at height zero.
\end{convention}

\begin{figure}[ht]
\centering
\begin{tikzpicture}[x=1cm,y=1cm,>=Stealth]
  \draw[panel] (-2.34,-1.02) rectangle (2.34,2.64);
  \draw[panel] (3.56,-1.02) rectangle (8.24,2.64);

  \node[vertex] (fa) at (0,1.35) {};
  \node[vertex] (fb) at (0,0) {};
  \draw[->,axis] (fa)--(fb);
  \node[lab,left] at (-.20,1.35) {\(z_{-1}\), \(h=k-1\)};
  \node[lab,left] at (-.20,0) {\(z_0\), \(h=k\)};
  \node[figuretitle,align=center] at (0,2.20) {forward state\\\((\fwd,k)\)};
  \node[callout] at (0,-.66) {height increases\\toward the cusp};

  \node[vertex] (ba) at (5.90,1.35) {};
  \node[vertex] (bb) at (5.90,0) {};
  \draw[->,axis] (ba)--(bb);
  \node[lab,right] at (6.10,1.35) {\(z_{-1}\), \(h=k+1\)};
  \node[lab,right] at (6.10,0) {\(z_0\), \(h=k\)};
  \node[figuretitle,align=center] at (5.90,2.20) {backward state\\\((\bwd,k)\)};
  \node[callout] at (5.90,-.66) {height decreases\\toward the compact part};
\end{tikzpicture}
\caption{The two oriented rooted states.  In both pictures the arrow is the incoming oriented edge \(z_{-1}\to z_0\).  The forward state moves from height \(k-1\) to height \(k\), while the backward state moves from height \(k+1\) to height \(k\).}
\label{fig:states}
\end{figure}
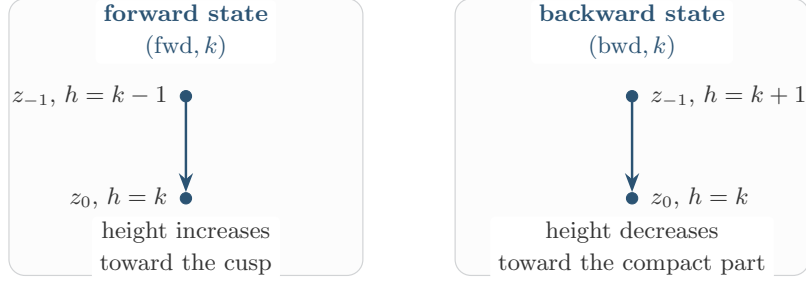

\begin{definition}[Descendant shadow]
For \(N\ge0\), the depth-\(N\) descendant shadow \(\Desc_N(e)\) is the set of
terminal vertices \(z_N\) of non-backtracking paths
\[
z_{-1}\to z_0\to z_1\to\cdots\to z_N.
\]
At each step there are exactly \(q\) allowed continuations.  Distinct
continuation sequences have distinct terminal vertices because \(\Tree\) is a
tree.  Hence
\[
|\Desc_N(e)|=q^N.
\]
Only the terminal generation is counted, as emphasized in
Figure~\ref{fig:terminal-shadow}.
\end{definition}

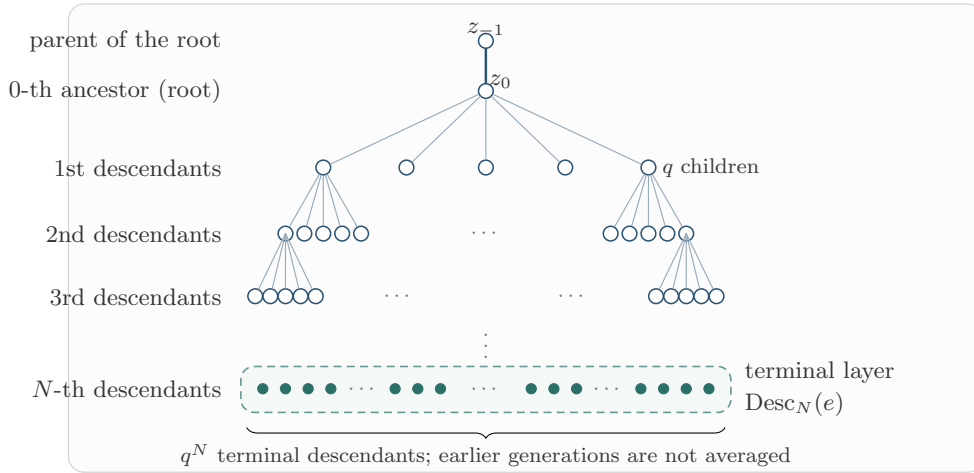
\begin{figure}[ht]
\centering
\begin{tikzpicture}[x=1cm,y=0.92cm,>=Stealth]
  \begin{scope}[on background layer]
    \draw[panel] (-5.52,-5.48) rectangle (6.58,1.25);
  \end{scope}
  \node[lab,anchor=east] at (-3.45,0.72) {parent of the root};
  \node[lab,anchor=east] at (-3.45,0.00) {$0$-th ancestor (root)};
  \node[lab,anchor=east] at (-3.45,-1.10) {1st descendants};
  \node[lab,anchor=east] at (-3.45,-2.05) {2nd descendants};
  \node[lab,anchor=east] at (-3.45,-2.95) {3rd descendants};
  \node[lab,anchor=east] at (-3.45,-4.28) {$N$-th descendants};

  \node[openvertex] (p) at (0,0.72) {};
  \node[openvertex] (r) at (0,0) {};
  \draw[axis] (p)--(r);
  \node[lab,above] at (p) {$z_{-1}$};
  \node[lab,above right] at (r) {$z_0$};

  \foreach \x/\name in {-2.15/a,-1.05/b,0/c,1.05/d,2.15/e}{
    \node[openvertex] (one\name) at (\x,-1.10) {};
    \draw[subtree] (r)--(one\name);
  }
  \node[smalllab,right] at (2.30,-1.10) {$q$ children};

  \foreach \x/\parent/\tag in {-2.65/a/la,-2.4/a/lb,-2.15/a/lc,-1.9/a/ld,-1.65/a/le,1.65/e/ra,1.9/e/rb,2.15/e/rc,2.4/e/rd,2.65/e/re}{
    \node[openvertex] (two\tag) at (\x,-2.05) {};
    \draw[subtree] (one\parent)--(two\tag);
  }
  \node[smallgraylab] at (0,-2.05) {$\cdots$};

  \foreach \x/\px/\tag in {-3.05/-2.65/a,-2.85/-2.65/b,-2.65/-2.65/c,-2.45/-2.65/d,-2.25/-2.65/e,2.25/2.65/f,2.45/2.65/g,2.65/2.65/h,2.85/2.65/i,3.05/2.65/j}{
    \node[openvertex] (three\tag) at (\x,-2.95) {};
    \draw[subtree] (\px,-2.05)--(three\tag);
  }
  \node[smallgraylab] at (-1.15,-2.95) {$\cdots$};
  \node[smallgraylab] at (1.15,-2.95) {$\cdots$};

  \node[smallgraylab] at (0,-3.55) {$\vdots$};

  \foreach \x in {-2.95,-2.65,-2.35,-2.05,-1.20,-.90,-.60,.60,.90,1.20,2.05,2.35,2.65,2.95}{
    \node[terminal] at (\x,-4.28) {};
  }
  \node[smallgraylab] at (-1.62,-4.28) {$\cdots$};
  \node[smallgraylab] at (0,-4.28) {$\cdots$};
  \node[smallgraylab] at (1.62,-4.28) {$\cdots$};

  \begin{scope}[on background layer]
    \draw[terminalband] (-3.24,-4.62) rectangle (3.24,-3.96);
  \end{scope}
  \node[lab,anchor=west,align=left] at (3.38,-4.28)
    {terminal layer\\$\Desc_N(e)$};
  \draw[decorate,decoration={brace,mirror,amplitude=4pt}]
    (-3.16,-4.84)--(3.16,-4.84)
    node[midway,below=5pt,smalllab]
    {$q^N$ terminal descendants; earlier generations are not averaged};
\end{tikzpicture}
\caption{A schematic rooted $q$-ary descendant tree.  The first few generations are drawn explicitly, intermediate generations are suppressed by dots, and the averaging set is the final horizontal layer of $N$-th descendants.}
\label{fig:terminal-shadow}
\end{figure}

\begin{definition}[Height counts and normalized laws]
Let \(\mathsf C^{\fwd}_{k,N}(j)\) be the number of vertices of height \(j\) in a
depth-\(N\) descendant shadow starting from a forward state \((\fwd,k)\),
where \(k\ge1\), and
let \(\mathsf C^{\bwd}_{k,N}(j)\) be the corresponding count for a backward state
\((\bwd,k)\), where \(k\ge0\).  Define normalized distributions by
\[
\mu^{\fwd}_{k,N}(j)=q^{-N}\mathsf C^{\fwd}_{k,N}(j),
\qquad
\mu^{\bwd}_{k,N}(j)=q^{-N}\mathsf C^{\bwd}_{k,N}(j).
\]
By Proposition~\ref{prop:local-height}, these counts depend only on the
oriented state and the root height, not on the chosen oriented edge.
The symbol \(\delta_j\) denotes the Dirac mass at height \(j\), whereas
\(\mathbf 1_E\) denotes the indicator of a condition \(E\).
\end{definition}

\begin{definition}[Admissible parity]
A pair \((k,N)\) has admissible even parity if \(k+N\) is even.  In that case all
terminal heights of the depth-\(N\) shadow from height \(k\) lie in
\(2\ZZ_{\ge0}\).  Equivalently, \(M=N-k\) is even.
\end{definition}

\subsection{The basic recursion}
\begin{lemma}[Forward/backward recursion]\label{lem:recursion}
For \(k\ge1\), \(N\ge1\), and \(j\ge0\),
\[
\mathsf C^{\bwd}_{k,N}(j)=q\,\mathsf C^{\bwd}_{k-1,N-1}(j),
\]
and
\[
\mathsf C^{\fwd}_{k,N}(j)=\mathsf C^{\fwd}_{k+1,N-1}(j)+(q-1)\mathsf C^{\bwd}_{k-1,N-1}(j).
\]
At the bottom, Convention~\ref{conv:bottom-backward} gives
\[
\mathsf C^{\bwd}_{0,N}(j)=q\,\mathsf C^{\fwd}_{1,N-1}(j)\qquad(N\ge1).
\]
\end{lemma}

\begin{proof}
From a backward state at height \(k\ge1\), all \(q\) non-backtracking
continuations move to height \(k-1\), and each remains backward type.  This gives
the first recursion.  From a forward state at height \(k\), one continuation moves
to the unique neighbor of height \(k+1\), and the remaining \(q-1\) continuations
move to height \(k-1\), where the state becomes backward type.  This gives the
second recursion.  The bottom formula follows because a backward state at height
zero has exactly \(q\) non-backtracking continuations, all entering the forward
state at height one.
\end{proof}

\begin{lemma}[Backward motion before the bottom]\label{lem:back-before-bottom}
If \(0\le N\le k\), then
\[
\mathsf C^{\bwd}_{k,N}(j)=q^N\mathbf 1_{\{j=k-N\}}.
\]
If \(N\ge k+1\), then
\[
\mathsf C^{\bwd}_{k,N}(j)=q^{k+1}\mathsf C^{\fwd}_{1,N-k-1}(j).
\]
\end{lemma}

\begin{proof}
Before reaching height zero, a backward state decreases the height by one at each
step and branches by the factor \(q\).  After \(k\) steps all descendants are at
height zero; the next step has \(q\) non-backtracking choices, all to height one,
and the remaining process is the distinguished forward process at height one.
\end{proof}

\subsection{The distinguished forward process}
\begin{proposition}[Explicit shell formula at height one]\label{prop:forward-one-shell}
For \(N\ge0\),
\[
\mathsf C^{\fwd}_{1,N}(j)=
\begin{cases}
1, & j=N+1,\\
(q^2-1)q^{N-j-1}, & 1\le j\le N-1,\ j\equiv N+1\pmod2,\\
(q-1)q^{N-1}, & j=0,\ N\text{ odd},\\
0, & \text{otherwise}.
\end{cases}
\]
\end{proposition}

\begin{proof}
Write \(C_N(j)=\mathsf C^{\fwd}_{1,N}(j)\).  We prove the formula by strong
induction on \(N\).  The case \(N=0\) is \(C_0(j)=\mathbf 1_{\{j=1\}}\).
For \(N\ge1\), separate the always-forward path and, for every other path, let
\(r+1\) be the time of its first turn.  Immediately after that turn the state
is backward at height \(r\), with \(N-r-1\) steps remaining.  Hence
\[
C_N(j)=\mathbf 1_{\{j=N+1\}}+(q-1)\sum_{r=0}^{N-1}
       \mathsf C^{\bwd}_{r,N-r-1}(j).
\]
Applying Lemma~\ref{lem:back-before-bottom} to each summand gives the renewal
identity
\begin{align}\label{eq:C-renewal}
C_N(j)
&=\mathbf 1_{\{j=N+1\}}
 +(q-1)\sum_{0\le r\le (N-2)/2}q^{r+1}C_{N-2r-2}(j) \notag\\
&\quad +(q-1)
 \sum_{\lceil(N-1)/2\rceil\le r\le N-1}
 q^{N-r-1}\mathbf 1_{\{j=2r-N+1\}},
\end{align}
where a sum over an empty index set is zero.  The two sums correspond,
respectively, to a backward segment that reaches height zero and one that ends
before reaching height zero.

Every term in \eqref{eq:C-renewal} has parity \(j\equiv N+1\pmod2\), and only
the first term can reach \(j=N+1\).  Now let
\(1\le j\le N-1\) have the admissible parity and put
\[
R=\frac{N-j-1}{2}.
\]
In the first sum of \eqref{eq:C-renewal}, the terms with \(r>R\) vanish; the
term \(r=R\) is the top term of \(C_{j-1}(j)\) and equals one.  For \(r<R\),
the induction hypothesis gives the interior coefficient.  The second sum has
exactly one nonzero term, at \(r=(N+j-1)/2\).  Therefore
\begin{align*}
C_N(j)
 &=(q-1)\left((q^2-1)\sum_{r=0}^{R-1}q^{2R-r-1}
                    +q^{R+1}+q^R\right)\\
 &=(q-1)(q+1)q^{2R}
  =(q^2-1)q^{N-j-1}.
\end{align*}
Here we used
\((q-1)\sum_{s=R}^{2R-1}q^s=q^{2R}-q^R\); this also covers \(R=0\).

Finally, the bottom shell can occur only when \(N=2R+1\) is odd.  Using the
induction hypothesis in the first sum and the unique term \(r=R\) in the
second gives
\[
C_N(0)
 =(q-1)^2\sum_{s=R}^{2R-1}q^s+(q-1)q^R
 =(q-1)q^{2R}=(q-1)q^{N-1}.
\]
All remaining heights have zero count by parity or by the displayed support
conditions.  This completes the induction.
\end{proof}

\begin{proposition}[First-turn decomposition]\label{prop:first-turn}
For every \(k\ge1\) and \(N\ge0\),
\[
\mathsf C^{\fwd}_{k,N}(j)=\mathbf 1_{\{j=k+N\}}+(q-1)\sum_{r=0}^{N-1}
\mathsf C^{\bwd}_{k+r-1,N-r-1}(j).
\]
Equivalently,
\begin{equation}\label{eq:first-turn-normalized}
\mu^{\fwd}_{k,N}=q^{-N}\delta_{k+N}
 +(q-1)\sum_{r=0}^{N-1}q^{-r-1}\mu^{\bwd}_{k+r-1,N-r-1}.
\end{equation}
\end{proposition}

\begin{proof}
There is one path that never leaves the forward geodesic; it ends at height
\(k+N\).  Every other path has a first backward turn at time \(r+1\).  Up to time
\(r\), it follows the unique forward geodesic and reaches height \(k+r\).  Then
it chooses one of \(q-1\) backward branches and the remaining \(N-r-1\) steps are
governed by the backward process at height \(k+r-1\).  If \(k=1\) and \(r=0\),
this uses the bottom state \((\bwd,0)\) of Convention~\ref{conv:bottom-backward}.
Normalization by \(q^N\) gives \eqref{eq:first-turn-normalized}.
\end{proof}

The two parts of \eqref{eq:first-turn-normalized} are displayed in
Figure~\ref{fig:first-turn}.

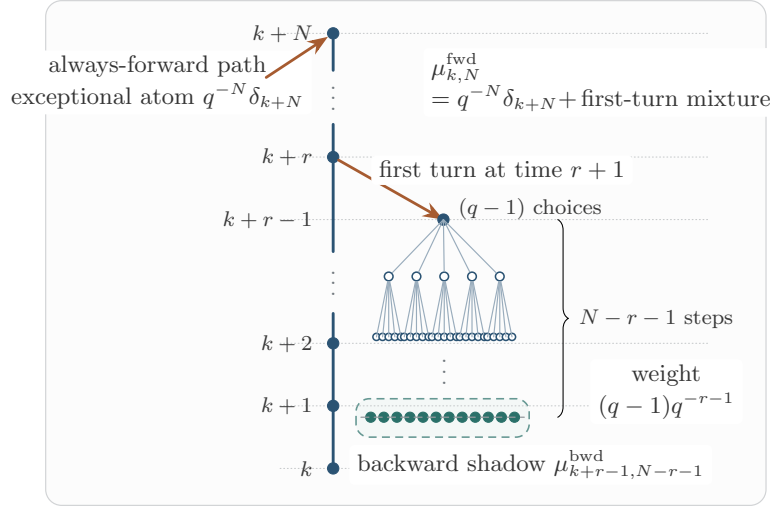
\begin{figure}[ht]
\centering
\begin{tikzpicture}[x=1.08cm,y=.94cm,>=Stealth]
  \begin{scope}[on background layer]
    \draw[panel] (-3.52,-.52) rectangle (5.30,6.58);
  \end{scope}
  \foreach \yy/\labtext in {
    0/{$k$},
    .88/{$k+1$},
    1.76/{$k+2$},
    3.50/{$k+r-1$},
    4.38/{$k+r$},
    6.12/{$k+N$}
  }{
    \draw[guide] (-.65,\yy)--(4.60,\yy);
    \node[smalllab,left] at (-.22,\yy) {\labtext};
  }

  \node[vertex] (vk)  at (0,0) {};
  \node[vertex] (vk1) at (0,.88) {};
  \node[vertex] (vk2) at (0,1.76) {};
  \node[vertex] (vkr) at (0,4.38) {};
  \node[vertex] (vkN) at (0,6.12) {};

  \draw[accent] (vk)--(vk1)--(vk2);
  \draw[accent] (vk2)--(0,2.18);
  \node[smallgraylab] at (0,2.72) {$\vdots$};
  \draw[accent] (0,3.05)--(vkr);
  \draw[accent] (vkr)--(0,4.84);
  \node[smallgraylab] at (0,5.32) {$\vdots$};
  \draw[accent] (0,5.60)--(vkN);

  \node[callout] at (-2.15,5.43)
    {always-forward path\\exceptional atom $q^{-N}\delta_{k+N}$};
  \draw[incoming] (-.82,5.50)--(vkN);

  \coordinate (turn) at (1.35,3.50);
  \draw[incoming] (vkr)--(turn);
  \node[vertex] at (turn) {};

  \node[callout,anchor=south west] at (.46,3.98)
    {first turn at time $r+1$};
  \node[smalllab,right] at (1.50,3.66) {$(q-1)$ choices};

  \foreach \x/\name in {.68/a,1.02/b,1.36/c,1.70/d,2.04/e}{
    \node[openvertex,inner sep=1.15pt] (\name) at (\x,2.70) {};
    \draw[subtree] (turn)--(\name);
  }

  \foreach \x/\parent in {.68/a,1.02/b,1.36/c,1.70/d,2.04/e}{
    \foreach \dx in {-.15,-.075,0,.075,.15}{
      \node[openvertex,inner sep=.95pt] at (\x+\dx,1.85) {};
      \draw[subtree] (\parent)--(\x+\dx,1.85);
    }
  }

  \node[smallgraylab] at (1.36,1.47) {$\vdots$};

  \foreach \x in {.46,.62,.78,.94,1.10,1.26,1.42,1.58,1.74,1.90,2.06,2.22}{
    \node[terminal] at (\x,.72) {};
  }
  \draw[horoline] (.34,.72)--(2.34,.72);
  \begin{scope}[on background layer]
    \draw[terminalband] (.28,.44) rectangle (2.40,.98);
  \end{scope}

  \node[callout] at (2.42,.05)
    {backward shadow $\mu^{\bwd}_{k+r-1,N-r-1}$};

  \draw[decorate,decoration={brace,amplitude=4pt}]
    (2.75,3.50)--(2.75,.72)
    node[midway,right=7pt,smalllab]
    {$N-r-1$ steps};

  \node[callout] at (4.08,1.08)
    {weight\\$(q-1)q^{-r-1}$};

  \node[callout,align=left] at (3.28,5.43)
    {$\mu^{\fwd}_{k,N}$\\
     $=q^{-N}\delta_{k+N}+\,$first-turn mixture};
\end{tikzpicture}
\caption{Vertical first-turn decomposition.  The forward spine is followed through
heights \(k,k+1,k+2,\ldots,k+r,\ldots,k+N\).  At height \(k+r\) the first
backward turn moves to height \(k+r-1\), and the remaining \(N-r-1\) steps form
the backward shadow \(\mu^{\bwd}_{k+r-1,N-r-1}\).}
\label{fig:first-turn}
\end{figure}

\subsection{The even limiting law and exact discrepancy}
We now restrict to admissible even parity.  For a height profile
\(\Phi\in L^1(\rho^{\ev})\), put
\[
\ang\Phi\rang_{\rho^{\ev}}=\sum_{m\ge0}\rho^{\ev}(2m)\Phi(2m).
\]

\begin{proposition}[Backward exact discrepancy]\label{prop:back-discrepancy}
Let \(k\ge0\) and \(N\ge k\), let the initial state be backward type at height
\(k\), and assume that \((k,N)\) has admissible parity.  Put
\(M=N-k\), and let
\(\Phi\in L^1(\rho^{\ev})\).  Then
\[
\int \Phi\,d\mu^{\bwd}_{k,N}-\ang\Phi\rang_{\rho^{\ev}}
=q^{-M-1}\Phi(M)-\sum_{2m\ge M+2}\rho^{\ev}(2m)\Phi(2m).
\]
Consequently,
\[
\left|\int \Phi\,d\mu^{\bwd}_{k,N}-\ang\Phi\rang_{\rho^{\ev}}\right|
\le \sum_{2m\ge M}\rho^{\ev}(2m)|\Phi(2m)|.
\]
\end{proposition}

\begin{proof}
If \(M=0\), then the backward distribution is \(\delta_0\), and the displayed
formula says
\[
\Phi(0)-\ang\Phi\rang_{\rho^{\ev}}
= q^{-1}\Phi(0)-\sum_{2m\ge2}\rho^{\ev}(2m)\Phi(2m),
\]
which follows from \(\rho^{\ev}(0)=(q-1)/q\).  If \(M\ge2\),
Lemma~\ref{lem:back-before-bottom} reduces the normalized backward process to a
shifted copy of the distinguished process of Proposition~\ref{prop:forward-one-shell}.  Its
weights agree exactly with \(\rho^{\ev}\) on the bottom shell and on all interior
shells below the top height \(M\).  At the top shell the difference is
\[
q^{-(M-1)}-(q^2-1)q^{-M-1}=q^{-M-1},
\]
and all shells above \(M\) are missing.  This proves the identity.  The bound
follows from \(q^{-M-1}\le \rho^{\ev}(M)\), with the case \(M=0\) included.
\end{proof}

\begin{proposition}[Forward exact discrepancy]\label{prop:forward-discrepancy}
Let \(k\ge1\) and \(N\ge0\), let the initial state be forward type at height
\(k\), assume admissible parity, and let \(\Phi\in L^1(\rho^{\ev})\).  Then
\begin{align*}
\int \Phi\,d\mu^{\fwd}_{k,N}-\ang\Phi\rang_{\rho^{\ev}}
&=q^{-N}\bigl(\Phi(k+N)-\ang\Phi\rang_{\rho^{\ev}}\bigr)  \\
&\quad +(q-1)\sum_{r=0}^{N-1}q^{-r-1}
\left(\int \Phi\,d\mu^{\bwd}_{k+r-1,N-r-1}-\ang\Phi\rang_{\rho^{\ev}}\right).
\end{align*}
\end{proposition}

\begin{proof}
This is obtained by integrating \(\Phi\) against the normalized first-turn
decomposition \eqref{eq:first-turn-normalized} and then subtracting
\(\ang\Phi\rang_{\rho^{\ev}}\), using
\((q-1)\sum_{r=0}^{N-1}q^{-r-1}=1-q^{-N}\).
\end{proof}

\begin{proposition}[Exact agreement on finite height windows]\label{prop:finite-window-exact}
Let \(\varepsilon\in\{\fwd,\bwd\}\), let \(k,N\ge0\), let \((k,N)\) have
admissible parity, and put
\(M=N-k\ge0\).  Assume \(k\ge1\) if \(\varepsilon=\fwd\).  Let
\(R\in2\ZZ_{\ge0}\).  If \(M>R\), then
\[
\mu^{\varepsilon}_{k,N}(j)=\rho^{\ev}(j)\qquad (j\in2\ZZ_{\ge0},\ 0\le j\le R).
\]
Consequently, if \(\Phi:2\ZZ_{\ge0}\to\CC\) is supported in \(\{0,2,\ldots,R\}\), then
\[
\int \Phi\,d\mu^{\varepsilon}_{k,N}=\ang\Phi\rang_{\rho^{\ev}}.
\]
\end{proposition}

\begin{proof}
For the backward state, the assertion follows immediately from the exact formula
of Proposition~\ref{prop:back-discrepancy}: applying the formula to the indicator
of an even shell \(j\le R<M\), neither the top-shell correction at \(M\) nor the
missing tail above \(M\) contributes.

We prove the forward case by the first-turn decomposition.  Fix an even height
\(j<M\).  The always-forward atom is at height \(k+N=M+2k>j\), so it does not
contribute to the shell \(j\).  In the first-turn sum, put
\[
M_r=(N-r-1)-(k+r-1)=M-2r.
\]
For \(r<(M-j)/2\), the backward subshadow has cutoff \(M_r>j\), and the backward
case gives its shell mass \(\rho^{\ev}(j)\).  For \(r=(M-j)/2\), the shell \(j\) is
the top shell of the backward subshadow, whose mass is \(\rho^{\ev}(j)+q^{-j-1}\)
(with the convention that this is \(1\) when \(j=0\)).  For \((M-j)/2<r<(M+j)/2\),
the backward subshadow has already truncated below height \(j\), so the shell
mass is zero.  Finally, when \(r=(M+j)/2\), the late first-turn path which has not
reached the bottom yet contributes a point mass at height \(j\); for \(j=0\) this
last case coincides with the top-shell case.  Summing these mutually exclusive
contributions with the weights \((q-1)q^{-r-1}\) gives \(\rho^{\ev}(j)\).  Indeed,
for \(j>0\), writing \(r_0=(M-j)/2\), the coefficient of \(\rho^{\ev}(j)\) from
\(r<r_0\) and \(r=r_0\) is \(1-q^{-r_0-1}\), while the two correction terms are
\[
(q-1)q^{-r_0-1}q^{-j-1}+(q-1)q^{-(M+j)/2-1}
=\rho^{\ev}(j)q^{-r_0-1}.
\]
For \(j=0\), the same computation has only the single correction at \(r=M/2\),
and it gives \(\rho^{\ev}(0)q^{-M/2-1}\), exactly the mass missing from the
geometric sum of the \(\rho^{\ev}(0)\)-terms.  Thus
\(\mu^{\fwd}_{k,N}(j)=\rho^{\ev}(j)\) for every even \(j<M\).  Since \(R<M\), the
claim follows by linearity for all profiles supported in \(\{0,2,\ldots,R\}\).
\end{proof}

\begin{lemma}[Uniform domination for a fixed root]\label{lem:fixed-domination}
Fix \(\varepsilon\in\{\fwd,\bwd\}\) and a height \(k\ge0\), with \(k\ge1\)
when \(\varepsilon=\fwd\).  There is a constant \(C_{q,k}\) such that
\[
\mu^\varepsilon_{k,N}(j)\le C_{q,k}\rho^{\ev}(j)
\]
for every \(N\ge0\) satisfying \(k+N\equiv0\pmod2\) and every even
\(j\ge0\).
\end{lemma}

\begin{proof}
For a backward shadow that has reached the bottom, the explicit shell formula
in Proposition~\ref{prop:back-discrepancy}, applied to shell indicators, shows
that its mass is \(\rho^{\ev}(j)\) below the cutoff, is at most a fixed
multiple of \(\rho^{\ev}(j)\) at the top shell, and is zero above it.  The
finitely many cases \(N<k\) are absorbed into a constant depending on \(k\).

For a forward shadow, use \eqref{eq:first-turn-normalized}.  The backward terms
whose cutoff is nonnegative are uniformly dominated as above.  If a late turn
has negative cutoff, Lemma~\ref{lem:back-before-bottom} makes that term a point
mass at
\[
j=2r-(N-k).
\]
Its coefficient is \((q-1)q^{-r-1}\).  Since
\(r\le N-1=(N-k)+k-1\), one has, for \(j>0\),
\[
\frac{(q-1)q^{-r-1}}{\rho^{\ev}(j)}
 =\frac{q^{j-r}}{q+1}
 =\frac{q^{r-(N-k)}}{q+1}
 \le\frac{q^{k-1}}{q+1}.
\]
Finally, the always-forward atom has mass \(q^{-N}\) at \(j=k+N\), and
\[
\frac{q^{-N}}{\rho^{\ev}(k+N)}
 =\frac{q^{k+1}}{q^2-1}.
\]
This proves the claim.
\end{proof}

\begin{theorem}[Rooted equidistribution]\label{thm:rooted-equidistribution}
If \(\Phi\in L^1(\rho^{\ev})\), then for every fixed \(k\ge0\),
\[
\int \Phi\,d\mu^{\bwd}_{k,N}\to \ang\Phi\rang_{\rho^{\ev}}
\]
as \(N\to\infty\) through admissible parity.
For every fixed \(k\ge1\), the same conclusion holds with
\(\mu^{\bwd}_{k,N}\) replaced by \(\mu^{\fwd}_{k,N}\).
\end{theorem}

\begin{proof}
The backward convergence is the tail estimate in
Proposition~\ref{prop:back-discrepancy}.  For either state,
Proposition~\ref{prop:finite-window-exact} shows that
\(\mu^\varepsilon_{k,N}(j)\to\rho^{\ev}(j)\) at every fixed even shell.
Lemma~\ref{lem:fixed-domination} gives
\[
|\Phi(j)|\mu^\varepsilon_{k,N}(j)
\le C_{q,k}|\Phi(j)|\rho^{\ev}(j).
\]
The right-hand side is summable, so dominated convergence on the discrete
height space proves the assertion.
\end{proof}

\subsection{Effective estimates}
\begin{proposition}[Backward effective estimate]\label{prop:effective-back}
Let \(k\ge0\) and \(N\ge k\), and assume
\(k+N\equiv0\pmod2\).  Let
\(\Phi:2\ZZ_{\ge0}\to\CC\) satisfy
\(|\Phi(2m)|\le Cq^{\alpha m}\) for every \(m\ge0\), where \(C>0\) and
\(\alpha<2\).  Then
\[
\left|\int \Phi\,d\mu^{\bwd}_{k,N}-\ang\Phi\rang_{\rho^{\ev}}\right|
\ll_{q,\alpha} C q^{-(2-\alpha)(N-k)/2}.
\]
\end{proposition}

\begin{proof}
Put \(M=N-k\).  The assumptions make \(M\) a nonnegative even integer, so
the growth condition and \(\alpha<2\) imply
\(\Phi\in L^1(\rho^{\ev})\).  Proposition~\ref{prop:back-discrepancy} therefore
applies and gives
\[
\left|\int \Phi\,d\mu^{\bwd}_{k,N}
      -\ang\Phi\rang_{\rho^{\ev}}\right|
 \le \sum_{2m\ge M}\rho^{\ev}(2m)|\Phi(2m)|.
\]
The formula \eqref{eq:rhoev} implies
\(\rho^{\ev}(2m)\ll_q q^{-2m}\), also at \(m=0\).  Hence the right-hand side
is at most a constant depending only on \(q\) times
\[
C\sum_{m\ge M/2}q^{-(2-\alpha)m}
 =\frac{C}{1-q^{-(2-\alpha)}}q^{-(2-\alpha)M/2},
\]
which is the asserted estimate.
\end{proof}

\begin{proposition}[Forward effective estimate]\label{prop:effective-forward}
Fix \(k\ge1\).  Let \(N\ge k\) satisfy
\(k+N\equiv0\pmod2\).  Let
\(\Phi:2\ZZ_{\ge0}\to\CC\) satisfy
\(|\Phi(2m)|\le Cq^{\alpha m}\) for every \(m\ge0\), with \(C>0\) and
\(\alpha<2\), and put
\[
\beta=\frac{2-\alpha}{2}.
\]
Then
\[
\left|\int \Phi\,d\mu^{\fwd}_{k,N}-\ang\Phi\rang_{\rho^{\ev}}\right|
\ll_{q,\alpha,k}
\begin{cases}
Cq^{-N/2}, & \alpha<1,\\
CNq^{-N/2}, & \alpha=1,\\
Cq^{-\beta N}, & 1<\alpha<2.
\end{cases}
\]
Equivalently, the exponential rate is governed by
\(\min(1/2,(2-\alpha)/2)\), with a factor linear in \(N\) at \(\alpha=1\)
(equivalently, logarithmic in the terminal-layer size \(q^N\)).
\end{proposition}

\begin{proof}
Use the first-turn formula of Proposition~\ref{prop:forward-discrepancy}.  Write
\(M=N-k\ge0\).  As in the preceding proof, the growth condition implies
\(\Phi\in L^1(\rho^{\ev})\), so that formula applies.  The expectation satisfies
\(|\ang\Phi\rang_{\rho^{\ev}}|\ll_{q,\alpha}C\).  The always-forward term is
therefore bounded by
\[
q^{-N}\bigl(|\Phi(k+N)|+|\ang\Phi\rang_{\rho^{\ev}}|\bigr)
 \ll_{q,\alpha,k} Cq^{-\beta N}.
\]

For turns with \(r\le M/2\), the backward subshadow has the nonnegative even
cutoff \(M-2r\).  Proposition~\ref{prop:effective-back} gives the total
early-turn contribution
\[
\ll C\sum_{r\le M/2}q^{-r}q^{-\beta(M-2r)}
= Cq^{-\beta M}\sum_{r\le M/2}q^{(2\beta-1)r}.
\]
This is \(O(Cq^{-M/2})\) if \(\alpha<1\), \(O(CMq^{-M/2})\) if \(\alpha=1\), and
\(O(Cq^{-\beta M})\) if \(1<\alpha<2\).

For turns with \(r>M/2\), the backward subshadow has not yet reached the
bottom.  Lemma~\ref{lem:back-before-bottom} makes its normalized law the point
mass at \(d=k+2r-N=2r-M\).  The part containing \(\Phi\) is therefore bounded
by
\[
C\sum_{r>M/2}q^{-r}q^{\alpha(k+2r-N)/2}
 = Cq^{-M/2}
  \sum_{\substack{2\le d\le N+k-2\\ d\equiv M\ (2)}}
  q^{(\alpha-1)d/2}
 \le Cq^{-M/2}
  \sum_{\substack{0\le d\le N+k-2\\ d\equiv M\ (2)}}
  q^{(\alpha-1)d/2},
\]
where an empty sum is zero.  The part
containing \(\ang\Phi\rang_{\rho^{\ev}}\) is
\(O_{q,\alpha}(Cq^{-M/2})\), since the remaining first-turn weights form a
geometric tail.  The displayed finite sum now gives the same three cases: a
bounded geometric sum for \(\alpha<1\), a linear factor for \(\alpha=1\), and a
top-term bound \(O(Cq^{-\beta M})\) for \(1<\alpha<2\).  The always-forward
term was already estimated above.  Since \(k\) is fixed, replacing
\(M=N-k\) by \(N\) only changes the implicit constant.  Combining the three
parts proves the proposition.
\end{proof}

\subsection{Moving roots and total variation}
The dense-orbit application requires a version with moving initial height.  Let
\(N\ge0\), let \(\varepsilon_N\in\{\fwd,\bwd\}\), and let \(k_N\ge0\) be a
height with admissible parity.  If \(\varepsilon_N=\fwd\), assume
\(k_N\ge1\).  In addition assume
\[
M_N=N-k_N\ge0.
\]

\begin{proposition}[Uniform moving-root estimate]\label{prop:moving-TV}
For the rooted shadow distribution \(\mu^{\varepsilon_N}_{k_N,N}\),
\[
\|\mu^{\varepsilon_N}_{k_N,N}-\rho^{\ev}\|_{\TV}
\ll_q q^{-M_N/2}.
\]
If \(\varepsilon_N=\bwd\), then the sharper estimate \(O_q(q^{-M_N})\) holds.
\end{proposition}

\begin{proof}
In the backward case, Proposition~\ref{prop:back-discrepancy} shows that the
signed difference from \(\rho^{\ev}\) consists of one top-shell correction and the
missing tail above height \(M_N\).  The total mass of these two pieces is
\(O(q^{-M_N})\).

Assume now that \(\varepsilon_N=\fwd\).  By the first-turn decomposition,
\[
\mu^{\fwd}_{k_N,N}-\rho^{\ev}
=q^{-N}(\delta_{k_N+N}-\rho^{\ev})
 +(q-1)\sum_{r=0}^{N-1}q^{-r-1}
 \bigl(\mu^{\bwd}_{k_N+r-1,N-r-1}-\rho^{\ev}\bigr).
\]
The first signed term has total-variation norm at most
\(2q^{-N}\le2q^{-M_N}\), so it is absorbed by
\(O(q^{-M_N/2})\).  For the summation, put
\[
M_r=(N-r-1)-(k_N+r-1)=M_N-2r.
\]
If \(r\le M_N/2\), then the corresponding backward subshadow has nonnegative
cutoff \(M_r\), and the backward estimate gives a total variation bound
\(O(q^{-M_r})\).  Hence the contribution of these early turns is
\[
\ll \sum_{r\le M_N/2}q^{-r}q^{-(M_N-2r)}\ll q^{-M_N/2}.
\]
If \(r>M_N/2\), we only use the trivial bound
\(\|\mu^{\bwd}_{k_N+r-1,N-r-1}-\rho^{\ev}\|_{\TV}\ll1\).  The total geometric
weight of these late turns is
\[
\sum_{r>M_N/2}q^{-r}\ll q^{-M_N/2}.
\]
Combining the three estimates proves the forward bound.
\end{proof}

\section{Expanding translates of compact unipotent orbits}\label{sec:compact}

\subsection{Compact-orbit normal form}
For a compact orbit \(Y\), we write \(\nu_Y\) for its normalized
\(U\)-invariant probability measure.
The following elementary normal form gives a canonical spherical
representative for every compact \(U\)-orbit.

\begin{proposition}[Spherical normal form for compact \(U\)-orbits]
\label{prop:compact-normal-form}
Let \(Y=\Gamma gU\).  The orbit \(Y\) is compact if and only if
\(g\infty\in\mathbb P^1(\Fq(t))\).  If \(Y\) is compact, its parameter
stabilizer
\[
\Lambda_U(Y):=\{s\in F:\Gamma gu(s)=\Gamma g\}
\]
is independent of the representative and of the base point chosen on \(Y\).
It has covolume \(q^{-2m-1}\) for a unique integer \(m=m(Y)\), and for this
integer
\[
(h_K)_*\nu_Y=(h_K)_*\nu_{Y_m},
\qquad Y_m:=\Gamma a_mU.
\]
More precisely, \(m=-v_\infty(r)\), where one can write
\[
\gamma^{-1}g=d(r)u(s_0)=a_m d(\eta)u(s_0)
\]
for some \(\gamma\in\Gamma\), \(\eta\in\cO^\times\), and \(s_0\in F\).
\end{proposition}

\begin{proof}
Suppose first that \(g\infty\) is rational.  Since \(A=\Fq[t]\) is a
principal ideal domain, a Bezout identity shows that
\(\Gamma=\mathrm{SL}_2(A)\) acts transitively on
\(\mathbb P^1(\Fq(t))\).  Choose \(\gamma\in\Gamma\) with
\(\gamma\infty=g\infty\).  Then \(\gamma^{-1}g\) fixes \(\infty\), and hence
\[
\gamma^{-1}g=d(r)u(s_0)
\]
for some \(r\in F^\times\) and \(s_0\in F\).  Write \(r=t^m\eta\) with
\(m=-v_\infty(r)\in\ZZ\) and \(\eta\in\cO^\times\).  The factor \(u(s_0)\)
is absorbed by the right \(U\)-orbit, while
\[
d(\eta)u(s)=u(\eta^2s)d(\eta),
\]
so
\[
Y=\Gamma a_m d(\eta)U=Y_m d(\eta).
\]
The stabilizer of \(Y_m\) in the additive \(U\)-parameter is
\[
\Lambda_m=\{s\in F:\Gamma a_mu(s)=\Gamma a_m\}=t^{-2m}A.
\]
It is a cocompact lattice in \(F\), with fundamental domain
\(t^{-2m-1}\cO\).  Hence \(Y_m\), and therefore \(Y=Y_md(\eta)\), is compact.
Right translation by \(d(\eta)\in K\) does not change \(h_K\), and the
parameter change \(s\mapsto\eta^2s\) preserves additive Haar measure.
Consequently,
\((h_K)_*\nu_Y=(h_K)_*\nu_{Y_m}\).

The parameter stabilizer of \(Y\) is
\[
\Lambda_U(Y)=\{s\in F:gu(s)g^{-1}\in\Gamma\}
          =\eta^{-2}t^{-2m}A.
\]
Replacing \(g\) by \(\gamma_0g\), \(\gamma_0\in\Gamma\), or by
\(gu(s_1)\), \(s_1\in F\), leaves this subgroup unchanged, using
\(\Gamma\gamma_0=\Gamma\) and the commutativity of \(U\).  Its covolume under
the normalization \(ds(\cO)=1\) is
\(q^{-2m-1}\).  This covolume is intrinsic to the \(U\)-orbit and recovers
\(m\); in particular, neither the choice of \(\gamma\) nor another normal-form
presentation can change \(m(Y)\).  This proves uniqueness.

Conversely, suppose that \(Y\) is compact.  Since \(\Gamma\) is discrete,
\(\Lambda_U(Y)\) is discrete.  The orbit map
\[
\Lambda_U(Y)\backslash F\longrightarrow Y,
\qquad \Lambda_U(Y)+s\longmapsto\Gamma gu(s),
\]
is the canonical homeomorphism from the stabilizer quotient onto the closed
orbit \(Y\).  Thus compactness of \(Y\) implies that \(\Lambda_U(Y)\) is
cocompact, so it contains some \(s\ne0\).  The nontrivial unipotent matrix
\(gu(s)g^{-1}\in\mathrm{SL}_2(A)\) fixes \(g\infty\).  Its unique fixed line
is the kernel of \(gu(s)g^{-1}-I\), a matrix with entries in \(\Fq(t)\); that
line is therefore \(\Fq(t)\)-rational.  Hence
\(g\infty\in\mathbb P^1(\Fq(t))\), completing the converse.
\end{proof}

\subsection{The exact shadow of a standard closed horosphere}
For \(j\in\ZZ\), let
\[
\widetilde v_j=
 \begin{bmatrix}t^j&0\\0&1\end{bmatrix}v_0
\]
be the vertices of the standard apartment, extended in both directions, and
let
\[
e^-=(\widetilde v_0,\widetilde v_{-1}).
\]
Since both \(\widetilde v_1\) and \(\widetilde v_{-1}\) project to height one,
\(e^-\) is a forward state at height one.

\begin{lemma}[Polynomial cylinders and a terminal layer]
\label{lem:standard-compact-shadow}
For \(d\ge1\), put
\[
\mathcal P_d=
 \left\{\sum_{j=1}^{2d-1}c_jt^j:c_j\in\Fq\right\}.
\]
Then
\[
\left\{a_{-d}u(P)v_0:P\in\mathcal P_d\right\}
   =\Desc_{2d-1}(e^-).
\]
In particular, the set on the left has \(q^{2d-1}\) vertices.
\end{lemma}

\begin{proof}
Put \(\pi=t^{-1}\), and let
\(e_1=(1,0)^{\mathsf T}\), \(e_2=(0,1)^{\mathsf T}\) be the standard basis of
\(F^2\).  In the lattice model, the vertices at distance \(2d\)
from \(v_0=[\cO^2]\) whose first edge is
\(\widetilde v_0\to\widetilde v_{-1}\) have the unique representatives
\[
L_b=\cO(\pi^{2d}e_1)+\cO(e_2+be_1),
\qquad b\in\pi\cO/\pi^{2d}\cO.
\]
Indeed, replacing \(b\) by \(b+\pi^{2d}c\), \(c\in\cO\), changes the second
generator by an \(\cO\)-multiple of the first, while equality of the two
lattices forces precisely this congruence.  Truncating the expansion of \(b\)
modulo \(\pi^r\cO\), for \(r=2,\ldots,2d\), gives the successive vertices on
the path; passing from \(r\) to \(r+1\) chooses one new residue in
\(\pi^r\cO/\pi^{r+1}\cO\simeq\Fq\).  The condition \(b\in\pi\cO\) fixes the
first edge to be \(\widetilde v_0\to\widetilde v_{-1}\).  Thus these lattices
are exactly the \(q^{2d-1}\) vertices of \(\Desc_{2d-1}(e^-)\), not merely a
set of the same cardinality.

Up to a projective scalar, the matrix \(a_{-d}u(P)\) is
\[
\begin{pmatrix}\pi^{2d}&\pi^{2d}P\\0&1\end{pmatrix},
\]
so it represents \(L_b\) with \(b=\pi^{2d}P\).  Multiplication by
\(\pi^{2d}\) identifies
\[
\mathcal P_d
 \longrightarrow \pi\cO/\pi^{2d}\cO,
\qquad P\longmapsto\pi^{2d}P,
\]
bijectively.  This proves the asserted equality of terminal layers.
\end{proof}

\begin{proposition}[Exact compact-orbit shadow]
\label{prop:compact-shadow}
For \(d\ge1\), the \(K\)-spherical pushforward of Haar probability on
\(Y_{-d}=\Gamma a_{-d}U\) is
\[
(h_K)_*\nu_{Y_{-d}}=\mu^{\fwd}_{1,\,2d-1}.
\]
Consequently, if \(Y\) is a compact \(U\)-orbit with the integer \(m(Y)\) from
Proposition~\ref{prop:compact-normal-form}, then for
\(n\ge0\) with \(n>m(Y)\),
\begin{equation}\label{eq:compact-exact-shadow}
(h_K)_*\lambda_{Y,n}^{\mathrm{per}}
 =\mu^{\fwd}_{1,\,2(n-m(Y))-1},
\end{equation}
where
\[
\lambda_{Y,n}^{\mathrm{per}}(f)
   :=\int_Y f(ya_{-n})\,d\nu_Y(y).
\]
\end{proposition}

\begin{proof}
For \(Y_{-d}\), the stabilizer in the \(U\)-parameter is
\[
a_d\Gamma a_{-d}\cap U=\{u(s):s\in t^{2d}A\}.
\]
Partition the compact parameter quotient \(t^{2d}A\backslash F\) into
\(\cO\)-cosets.  The resulting additive double quotient
\[
t^{2d}A\backslash F/\cO
\]
has the unique representatives \(\mathcal P_d\).  Indeed, reduction modulo
\(\cO\) removes the terms of degree at most zero in a Laurent expansion,
whereas reduction modulo \(t^{2d}A\) removes the polynomial terms of degree at
least \(2d\).  The uniquely surviving polynomial part has degrees
\(1,\ldots,2d-1\).  Normalized Haar measure gives equal mass
\(q^{-(2d-1)}\) to these \(q^{2d-1}\) classes.
Lemma~\ref{lem:standard-compact-shadow} identifies the corresponding lifted
vertices with a depth-\((2d-1)\) descendant layer from the forward state at
height one.  Applying the quotient height map, with any further
\(\Gamma\)-identifications counted with their lifted multiplicity, proves the
first formula.

For a general compact orbit, write \(Y=\Gamma gU\).  Proposition~\ref{prop:compact-normal-form} gives
\(\gamma^{-1}g=a_m d(\eta)u(s_0)\).  Absorb \(u(s_0)\) into the orbit
parameter.  Since \(d(\eta)\) commutes with \(a_{-n}\), the identities
\[
d(\eta)u(s)=u(\eta^2s)d(\eta),
\qquad
u(r)a_{-n}=a_{-n}u(t^{2n}r)
\]
give
\[
a_m d(\eta)u(s)a_{-n}
 =a_{m-n}u(t^{2n}\eta^2s)d(\eta).
\]
The change of parameter \(s\mapsto t^{2n}\eta^2s\) carries normalized Haar
probability on the compact parameter quotient to normalized Haar probability
on the new quotient, and the final factor \(d(\eta)\in K\) is invisible in
spherical projection.  Hence the spherical translate is the normalized Haar
measure on \(Y_{m-n}\).  Taking \(d=n-m>0\) gives
\eqref{eq:compact-exact-shadow}.
\end{proof}

\begin{remark}[Why the inverse direction is necessary]
For the standard orbit \(Y_0=\Gamma U\), write
\(s=p+r\) with \(p\in A\) and \(r\in t^{-1}\cO\).  Since \(u(p)\in\Gamma\) and
\[
a_n^{-1}u(r)a_n=u(t^{-2n}r)\in K,
\]
one has
\[
\Gamma u(s)a_nK=\Gamma a_nK.
\]
Thus \(Y_0a_n\) has a point-mass spherical projection and moves into the cusp;
it is \(Y_0a_{-n}\) that expands on the right quotient.
\end{remark}

\subsection{Equidistribution and finite-window exactness}
\begin{theorem}[\(K\)-spherical equidistribution of compact-orbit translates]
\label{thm:compact-eq}
Let \(Y\) be a compact right \(U\)-orbit and let \(m=m(Y)\).  For every bounded
right \(K\)-invariant function \(f=\Phi\circ h_K\) and every
\(n\in\ZZ_{\ge0}\) with \(n>m\),
\[
\left|\lambda_{Y,n}^{\mathrm{per}}(f)-\int_Xf\,d\mu_X\right|
 \ll_q \|\Phi\|_\infty q^{-(n-m-1)}.
\]
In particular,
\[
\lambda_{Y,n}^{\mathrm{per}}(f)\longrightarrow\int_Xf\,d\mu_X.
\]
\end{theorem}

\begin{proof}
In \eqref{eq:compact-exact-shadow}, the root height is \(1\), the depth is
\(2(n-m)-1\), and the cutoff is \(2(n-m)-2\).  The assertion follows from
Proposition~\ref{prop:moving-TV} and Proposition~\ref{prop:haar-rho}.
\end{proof}

\begin{theorem}[Eventual exactness for compactly supported spherical tests]
\label{thm:compact-support-exact}
Let \(Y\) and \(m=m(Y)\) be as above.  Suppose that
\(f=\Phi\circ h_K\) is right \(K\)-invariant and that
\(R\in2\ZZ_{\ge0}\) and
\(\operatorname{supp}\Phi\subset\{0,2,\ldots,R\}\).  If
\(n\in\ZZ_{\ge0}\) and
\[
2(n-m)-2>R,
\]
then
\[
\lambda_{Y,n}^{\mathrm{per}}(f)=\int_Xf\,d\mu_X.
\]
\end{theorem}

\begin{proof}
Combine the exact shadow identity \eqref{eq:compact-exact-shadow} with
Proposition~\ref{prop:finite-window-exact}.
\end{proof}

\begin{remark}
The general equidistribution of expanding compact horospherical orbits belongs
to the framework of \cite{CFS2021}.  The content specific to the present
\(K\)-spherical Nagao model is the exact identity
\eqref{eq:compact-exact-shadow}, its finite-scale discrepancy, and the
eventual finite-window equality.
\end{remark}

\section{Dense unipotent orbits and standard F{\o}lner balls}
\label{sec:dense}

\subsection{A canonical horospherical truncation}
Let \(x=\Gamma g\in X\), and let \(\bar g\) be the image of \(g\) in
\(\bG\).  Using the elements \(\bar a_N\) defined in
Section~\ref{sec:spherical-dictionary}, put
\[
y_N=\bar g\bar a_Nv_0.
\]
Then \((y_0,y_1,\ldots)\) is the geodesic ray from
\(y_0=\bar g v_0\) to the boundary point
\[
\xi_g=g\infty.
\]
Its quotient-height sequence and incoming rooted edge at scale \(N\) are
\[
k_N(x)=h(y_N),
\qquad
e_N=(y_{N+1},y_N).
\]
We define the state \(\varepsilon_N(x)\) by
\[
\varepsilon_N(x)=
\begin{cases}
\bwd,&h(y_{N+1})=k_N(x)+1,\\
\fwd,&h(y_{N+1})=k_N(x)-1.
\end{cases}
\]
At height zero only the first case occurs, and it is interpreted by
Convention~\ref{conv:bottom-backward}.  Since
\(d_\Tree(y_0,y_N)=N\) and \(\bar g\), as the projective image of an element
of \(\mathrm{SL}_2(F)\), preserves bipartite type, the quotient height has
parity
\[
k_N(x)\equiv N\pmod2.
\]
Thus the rooted shadow from \(e_N\) has admissible even parity.  We put
\[
M_N(x)=N-k_N(x).
\]
This number can be negative for an initial finite range of scales, but it will
eventually be nonnegative for every dense orbit by
Corollary~\ref{cor:MN-diverges}.
When \(x\) is fixed, we abbreviate
\(k_N=k_N(x)\), \(\varepsilon_N=\varepsilon_N(x)\), and
\(M_N=M_N(x)\).

With the normalization \(ds(\cO)=1\) fixed above, define
\[
U_N=\{u(s):s\in t^N\cO\},
\qquad
\lambda_{x,N}^{\mathrm{den}}(f)
 =q^{-N}\int_{t^N\cO}f(xu(s))\,ds.
\]
Since \(ds(t^N\cO)=q^N\), this is a probability measure.  The sequence
\((U_N)\) is a compact-open F{\o}lner sequence.  In fact it has the stronger
exact-invariance property that, for
each fixed \(s_0\in F\), one has
\[
s_0+t^N\cO=t^N\cO
\]
for all sufficiently large \(N\).  More generally, the same equality holds
simultaneously for every \(s_0\) in a prescribed compact subset of \(F\).
This uniform-on-compacts exact invariance is the only F{\o}lner property used
below.

\begin{proposition}[Exact terminal-layer realization]
\label{prop:dense-exact-shadow}
For every \(x\in X\) and \(N\ge0\), the map
\[
t^N\cO/\cO\longrightarrow \Desc_N(e_N),
\qquad
s+\cO\longmapsto \bar g u(s)v_0,
\]
is a bijection.  Consequently,
\begin{equation}\label{eq:dense-exact-shadow}
(h_K)_*\lambda_{x,N}^{\mathrm{den}}
 =\mu^{\varepsilon_N(x)}_{k_N(x),N}.
\end{equation}
Possible identifications after quotienting by \(\Gamma\) are counted with their
multiplicity in the lifted terminal layer.
\end{proposition}

\begin{proof}
Apply the tree automorphism \(\bar g^{-1}\).  It is enough to consider the
standard ray \((v_0,v_1,\ldots)\) pointing to \(\infty\).  Direct conjugation
shows that
\[
u(s)v_j=v_j\quad\Longleftrightarrow\quad s\in t^j\cO.
\]
Indeed,
\(\bar a_j^{-1}u(s)\bar a_j=u(t^{-j}s)\in\bK\) exactly under this condition.
Since \(t^N\cO\subset t^{N+1}\cO\), every \(s\in t^N\cO\) fixes both \(v_N\)
and \(v_{N+1}\), and it sends the
standard non-backtracking segment
\[
v_{N+1}\to v_N\to v_{N-1}\to\cdots\to v_0
\]
to an allowed segment with terminal vertex \(u(s)v_0\).  Two parameters give
the same terminal vertex precisely when
\(u(s-s')v_0=v_0\), equivalently \(s-s'\in\cO\).  We have therefore obtained
an injection
\[
t^N\cO/\cO\hookrightarrow\Desc_N(v_{N+1}\to v_N).
\]
Its source has \(q^N\) elements, while the target has \(q^N\) elements by the
definition of a descendant shadow; hence the injection is a bijection.
Equivalently, the successive quotients
\(t^j\cO/t^{j-1}\cO\simeq\Fq\) record the \(q\) choices at each descendant
step.  Translating back by \(\bar g\) proves the first assertion.

The \(q^N\) cosets modulo \(\cO\) have equal Haar measure, and a right
\(K\)-invariant function is constant on each coset.  Taking quotient heights
therefore gives exactly the normalized rooted law in
\eqref{eq:dense-exact-shadow}.
\end{proof}

Figure~\ref{fig:folner-shadow} summarizes this re-rooting of the horospherical
ball as a terminal descendant layer.

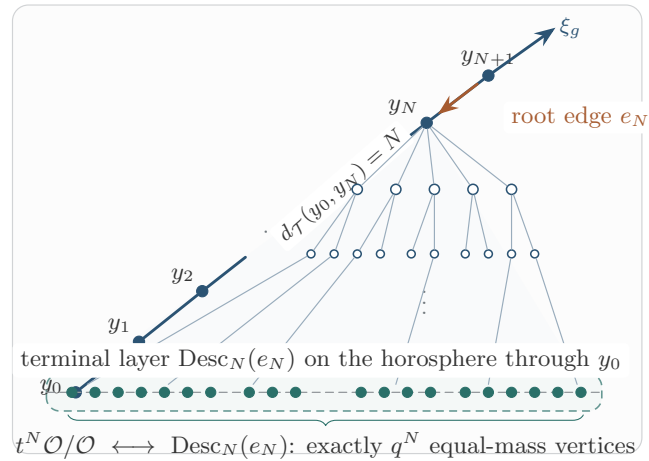
\begin{figure}[ht]
\centering
\begin{tikzpicture}[x=1.02cm,y=.88cm,>=Stealth]
  \begin{scope}[on background layer]
    \draw[panel] (-.82,-.92) rectangle (7.34,5.82);
    \fill[figblue!3] (4.55,4.05)--(.05,.20)--(6.55,.20)--cycle;
    \draw[terminalband] (-.38,-.28) rectangle (6.82,.28);
  \end{scope}

  \node[vertex] (y0) at (0,0) {};
  \node[lab,anchor=east] at (-.12,.08) {$y_0$};
  \node[vertex,label={[lab]above left:$y_1$}] (y1) at (.82,.76) {};
  \node[vertex,label={[lab]above left:$y_2$}] (y2) at (1.64,1.52) {};
  \node[vertex,label={[lab]above left:$y_N$}] (yN) at (4.55,4.05) {};
  \node[vertex,label={[lab]above:$y_{N+1}$}] (yNp) at (5.35,4.76) {};
  \draw[accent] (y0)--(y1)--(y2)--(2.20,2.03);
  \node[smallgraylab] at (2.62,2.40) {$\cdots$};
  \draw[accent] (3.02,2.70)--(yN)--(yNp);
  \draw[->,accent] (yNp)--(6.22,5.48)
    node[right,lab,font=\scriptsize\bfseries,text=figblue] {$\xi_g$};
  \draw[incoming]
    ($(yNp)!.18!(yN)$)--($(yNp)!.82!(yN)$);
  \node[callout,anchor=west,text=figorange]
    at (5.54,4.13) {root edge $e_N$};
  \node[smalllab,rotate=41,fill=white,rounded corners=1.5pt,
    inner xsep=2.2pt,inner ysep=1.4pt] at (3.45,3.10)
    {$d_\Tree(y_0,y_N)=N$};

  \foreach \x/\name in {3.65/a,4.15/b,4.65/c,5.15/d,5.65/e}{
    \node[openvertex,inner sep=1.35pt] (b\name) at (\x,3.05) {};
    \draw[subtree] (yN)--(b\name);
  }
  \foreach \x/\parent/\name in {
    3.05/a/f,3.35/a/g,3.65/b/h,3.95/b/i,
    4.35/c/j,4.65/c/k,5.05/d/l,5.35/d/m,5.65/e/n,5.95/e/o}{
    \node[openvertex,inner sep=1.05pt] (c\name) at (\x,2.08) {};
    \draw[subtree] (b\parent)--(c\name);
  }
  \node[smallgraylab] at (4.52,1.47) {$\vdots$};
  \draw[subtree] (cf)--(.15,.12);
  \draw[subtree] (cg)--(1.20,.12);
  \draw[subtree] (ch)--(2.30,.12);
  \draw[subtree] (ck)--(4.15,.12);
  \draw[subtree] (cn)--(5.35,.12);
  \draw[subtree] (co)--(6.50,.12);

  \draw[horoline] (-.38,0)--(6.82,0);
  \foreach \x in {-.05,.25,.55,.85,1.15,1.45,1.75,2.25,2.55,2.85,
                   3.70,4.00,4.30,4.60,5.05,5.35,5.65,5.95,6.25,6.55}{
    \node[terminal] at (\x,0) {};
  }
  \node[callout] at (3.18,.48)
    {terminal layer $\Desc_N(e_N)$ on the horosphere through $y_0$};
  \draw[decorate,decoration={brace,mirror,amplitude=4pt},draw=figteal]
    (-.10,-.34)--(6.58,-.34)
    node[midway,below=5pt,lab,text=black!80]
    {$t^N\cO/\cO\ \longleftrightarrow\ \Desc_N(e_N)$:
     exactly $q^N$ equal-mass vertices};
\end{tikzpicture}
\caption{The standard F{\o}lner ball becomes a depth-\(N\) terminal layer after
re-rooting at \(y_N\).  The root state is determined by the direction of the
projected geodesic at scale \(N\).}
\label{fig:folner-shadow}
\end{figure}

\subsection{Dense orbits and continued-fraction excursions}
The value \(\xi_g=g\infty\) depends on the lift \(g\), but replacing \(g\) by
\(\gamma g\), \(\gamma\in\Gamma\), applies \(\gamma\) to the boundary point and
does not change the quotient-height sequence \(k_N(x)\).  Rationality of
\(\xi_g\), and the tail of its continued-fraction coding, are therefore
intrinsic to the orbit.

\begin{lemma}[Compact versus dense unipotent orbits]
\label{lem:irrational-dense}
The right \(U\)-orbit \(xU\) is compact if and only if
\(\xi_g\in\mathbb P^1(\Fq(t))\).  If
\(\xi_g\notin\mathbb P^1(\Fq(t))\), then \(xU\) is dense.
\end{lemma}

\begin{proof}
The rational case, including compactness, is
Proposition~\ref{prop:compact-normal-form}.  We prove the converse without
changing from \(U\) to the larger group \(H_\infty\).

Apply the inversion map introduced in Section~\ref{sec:spherical-dictionary}:
\[
\iota(xU)=\{u^{-1}g^{-1}\Gamma:u\in U\}=Ug^{-1}\Gamma
       \subset G/\Gamma.
\]
The conjugate \(gUg^{-1}\) is a maximal unipotent subgroup whose unique fixed
point in \(\mathbb P^1(F)\) is \(g\infty=\xi_g\).  Every proper
\(\Fq(t)\)-parabolic subgroup of \(\mathrm{SL}_2\) is the stabilizer of a
unique \(\Fq(t)\)-rational line.  It follows that
\[
gUg^{-1}\ \hbox{is contained in a proper \(\Fq(t)\)-parabolic}
\quad\Longleftrightarrow\quad
\xi_g\in\mathbb P^1(\Fq(t)).
\]

The cited corollary identifies the limiting homogeneous measure through the
minimal rational parabolic containing the conjugated horospherical subgroup.
In rank one, irrationality of its fixed point rules out every proper rational
parabolic.  We now verify the quotient and averaging conventions needed to
apply it.

Suppose now that \(\xi_g\) is irrational.  We spell out the hypotheses of
\cite[Corollary~4.2]{Mohammadi2011}.  In its global notation, take
\[
F_{\mathrm{glob}}=\Fq(t),\qquad S=\{\infty\},\qquad \mathcal O_S=A,
\]
so that the resulting \(S\)-arithmetic quotient is \(G/\Gamma\).  Here
\(\Gamma=\mathrm{SL}_2(\mathcal O_S)\) is a congruence lattice.  The eigenvalues
of \(a_1=\operatorname{diag}(t,t^{-1})=
\operatorname{diag}(\pi^{-1},\pi)\) are integral powers of the uniformizer
\(\pi\), so \(a_1\) is an element of class~A in the terminology of the cited
paper.  Moreover, its maximal expanding horospherical subgroup is
\[
W_G^+(a_1):=\{h\in G:a_1^nha_1^{-n}\longrightarrow e
                  \text{ as }n\longrightarrow-\infty\}=U.
\]
Thus the hypotheses of the nonuniform case of
\cite[Theorem~1.1(ii)]{Mohammadi2011}, used in Corollary~4.2 there, are
satisfied.  At the left
quotient point \(h\Gamma=g^{-1}\Gamma\), the parabolic condition in that
corollary concerns
\[
h^{-1}Uh=gUg^{-1}\subset{}^\circ P.
\]
Here, following the cited paper, if \(P=LW\) is a Levi decomposition, then
\({}^\circ P=[L,L](F)W(F)\).
For a proper rank-one parabolic, \({}^\circ P\) is its unipotent radical, so
this containment condition is equivalent to rationality of the fixed line.
By the preceding rank-one observation, irrationality of \(\xi_g\) forces the
minimal \(F_{\mathrm{glob}}\)-parabolic \(P\) with this property to be \(G\)
itself.  The
symbol \(G^+\) in the conclusion of the cited result denotes the subgroup
generated by the \(F\)-split one-parameter unipotent subgroups of \(G\).  In
the present case \(G^+=G=\mathrm{SL}_2(F)\), since the upper and lower root
subgroups generate \(\mathrm{SL}_2(F)\).

Taking the diagonal element \(a_1\) and
\(U_0=\{u(s):s\in\cO\}\), the averaging sets in the cited theorem are
precisely the even subsequence of our balls:
\[
a_mU_0a_{-m}
 =\{u(s):s\in t^{2m}\cO\}.
\]
Their Haar mass is \(q^{2m}\), in agreement with the normalization in
\(\lambda_{x,2m}^{\mathrm{den}}\).
Under inversion, the average of \(xu(s)\) becomes the left average of
\(u(-s)g^{-1}\Gamma\); the sign is immaterial because
\(t^{2m}\cO=-t^{2m}\cO\).  Corollary~4.2 therefore gives weak-* convergence
of these even-subsequence averages to Haar probability on \(G/\Gamma\).
Let \(C=\overline{Ug^{-1}\Gamma}\).  Every averaging measure is supported on
the closed set \(C\), so its weak-* limit is supported on \(C\) as well.
Indeed, every function in \(C_c(G/\Gamma\setminus C)\) has zero integral
against each averaging measure and hence against the limit.
Since Haar probability has full support, this forces \(C=G/\Gamma\).  Thus the
left \(U\)-orbit is dense, and applying \(\iota^{-1}\) proves that \(xU\) is
dense in \(\Gamma\backslash G\).
\end{proof}

For \(\xi\in F\setminus\Fq(t)\), we use the Artin continued fraction
\[
\xi=[A_0;A_1,A_2,\ldots].
\]
Explicitly, \(\xi_0=\xi\),
\[
A_r=\lfloor\xi_r\rfloor,\qquad
\xi_{r+1}=(\xi_r-A_r)^{-1},
\]
and
\[
d_r=\deg A_r\ge1\qquad(r\ge1).
\]
The fractional-linear action of \(\Gamma\) may change the displayed digits,
but it changes only a finite initial part of the associated quotient-ray
excursion sequence.  In particular, the eventual sequence of their degrees is
well defined up to reindexing; this is the function-field modular-ray form of
the continued-fraction tail equivalence in
\cite[Theorem~3.2 and Proposition~3.5]{BroisePaulin2007}.

We also make explicit how the cited geodesic coding is used here.  Its
canonical first-return section is formulated for a bi-infinite oriented
geodesic, oriented so that its future endpoint is \(\xi\); the future return labels are the
Artin digits \(A_1,A_2,\ldots\).  After choosing one return vertex, we call the
positive half of this coded geodesic the \emph{standard Artin ray} associated
with \(\xi\).  This ray is auxiliary: every other geodesic ray with endpoint
\(\xi\) eventually coalesces with it in the tree.

The next proposition first states the exact geometry without an indexing
ambiguity and then identifies it with the continued-fraction degrees.  Whenever
the ray has infinitely many returns, discard a possible finite initial segment
and denote by
\[
0\le\tau_1<\tau_2<\cdots
\]
the successive return times for which
\[
k_{\tau_r}(x)=0,\qquad k_{\tau_r+1}(x)=1.
\]
Let \(d_r^{\mathrm{geo}}\) be the maximum height during the excursion from
\(\tau_r\) to \(\tau_{r+1}\).

\begin{proposition}[Return-time and continued-fraction coding]
\label{prop:cf-coding}
If \(\xi_g\) is irrational, the return times \(\tau_r\) form an infinite
sequence.  For every \(r\),
\[
\tau_{r+1}=\tau_r+2d_r^{\mathrm{geo}}.
\]
After deleting and relabeling finitely many initial continued-fraction digits,
\[
d_r^{\mathrm{geo}}=\deg A_r=:d_r.
\]
Writing \(B_r=\tau_r\), one therefore has
\[
B_{r+1}=B_r+2d_r.
\]
More precisely, during the ascending half of the \(r\)-th excursion,
\[
N=B_r+j,\quad k_N=j,\quad
\varepsilon_N=\bwd,\quad M_N=B_r
\qquad(0\le j\le d_r-1),
\]
whereas during the descending half,
\[
N=B_r+d_r+j,\quad k_N=d_r-j,\quad
\varepsilon_N=\fwd,\quad M_N=B_r+2j
\qquad(0\le j\le d_r-1).
\]
At \(N=B_{r+1}\), the next backward block begins and
\(M_N=B_{r+1}\).
\end{proposition}

\begin{proof}
By Proposition~\ref{prop:local-height}, a projected geodesic leaving height zero
must first ascend.  While ascending, it either continues through the unique
edge toward the cusp or turns through one of the other \(q-1\) edges.  Once it
turns, non-backtracking forces it to descend by one at every step until it
returns to height zero.  Hence an excursion of maximal height
\(d_r^{\mathrm{geo}}\) has length \(2d_r^{\mathrm{geo}}\), and substitution in
\(M_N=N-k_N\) gives the displayed state and cutoff formulas with
\(d_r^{\mathrm{geo}}\) in place of \(d_r\).

If there were only finitely many returns, there could be no later downward
turn: by the local height rule, every such turn would force the ray to descend
to height zero.  The quotient height would therefore increase by one forever.
Choose a lift of the first edge in this eventual tail and an element of
\(\GammaP\) carrying it to an edge \(v_j\to v_{j+1}\) of the standard cusp
ray.  At every positive height the continuation toward the cusp is unique, so
the entire lifted tail is carried to
\(v_j\to v_{j+1}\to v_{j+2}\to\cdots\).  Its endpoint is consequently a
\(\GammaP\)-translate of \(\infty\), hence lies in
\(\mathbb P^1(\Fq(t))\).  Irrationality therefore gives infinitely many
returns.

It remains to recall the arithmetic identification.  The first-return theorem
and digit identification in
\cite[Theorem~3.2 and Proposition~3.5]{BroisePaulin2007} associate one
nonconstant Artin digit \(A\) with the connected geodesic segment between two
successive cross-section returns.  In the notation of the cited proposition,
write
\[
i=\begin{bmatrix}0&1\\1&0\end{bmatrix},\qquad
t_A=u(A),\qquad
\lambda_\alpha=
 \begin{bmatrix}\alpha&0\\0&\alpha^{-1}\end{bmatrix}
 \quad(\alpha\in\Fq^\times).
\]
All these matrices are understood projectively.  The forward normalizing
elements in \cite[Proposition~3.5]{BroisePaulin2007} satisfy the following
formula.  Let \(\gamma_r^{\mathrm{BP}}\in\GammaP\) denote the normalizing
element attached there to the \(r\)-th forward return, and let
\(\alpha_r\in\Fq^\times\) be the corresponding unit.  Then
\[
(\gamma_{r-1}^{\mathrm{BP}})^{-1}\gamma_r^{\mathrm{BP}}
 =\lambda_{\alpha_{r-1}}^{-1}i\,t_{A_r}\lambda_{\alpha_r}.
\]
By construction, \(\gamma_r^{\mathrm{BP}}v_0\) is the \(r\)-th section-return
vertex, so the distance between consecutive return vertices is
\[
d_\Tree\bigl(v_0,
 (\gamma_{r-1}^{\mathrm{BP}})^{-1}\gamma_r^{\mathrm{BP}}v_0\bigr).
\]
Since \(i,\lambda_\alpha\in\bK\), the relative transition lies in the double
coset \(\bK u(A_r)\bK\).  If \(d=\deg A_r\), the two elementary-divisor
valuations of \(u(A_r)\) are \(-d\) and \(d\); equivalently,
\[
\bK u(A_r)\bK=\bK\bar a_{2d}\bK.
\]
Consequently,
\[
d_{\Tree}(\gamma_{r-1}^{\mathrm{BP}}v_0,
          \gamma_r^{\mathrm{BP}}v_0)
 =d_{\Tree}(v_0,u(A_r)v_0)=2d.
\]
By the definition of successive returns in
\cite[Theorem~3.2]{BroisePaulin2007}, there is no intermediate section return
inside this segment.  Thus this displacement is one complete cusp excursion
rather than a sum of several excursions.  The local height rule of
Proposition~\ref{prop:local-height} then forces the complete quotient-height
word to be
\[
0,1,2,\ldots,d-1,d,d-1,\ldots,2,1,0.
\]
Indeed, before the unique turn there is only one edge that continues toward
the cusp, and after the turn non-backtracking permits only decreasing height
until the bottom is reached.  Thus the excursion has maximum \(d\), not merely
total displacement \(2d\).

Finally, the coding theorem is stated for the bi-infinite geodesic and its
first-return section, whereas the present construction begins with the
one-sided ray based at \(\bar g v_0\).  Its endpoint is the same point
\(g\infty\) as that of the standard Artin ray defined above.  Two geodesic rays
in a tree with the same endpoint coalesce after a finite initial segment, so
the two rays have identical return segments after finitely many edges.  Their
excursion sequences differ only by a finite prefix, which accounts for the
permitted deletion and relabeling of the digits.  Applying the displayed
height word with \(d=d_r\) gives
\(B_{r+1}-B_r=2d_r\), and the two halves give exactly the stated formulas for
\(k_N\), \(\varepsilon_N\), and \(M_N=N-k_N\).
\end{proof}

\begin{corollary}[Monotonicity and divergence of the cutoff]
\label{cor:MN-diverges}
With the abbreviations fixed above, the sequence \(M_N=N-k_N\) is
nondecreasing: it stays constant when the
projected ray ascends and increases by two when the ray descends.  If
\(\xi_g\) is irrational, then
\[
M_N\longrightarrow\infty.
\]
\end{corollary}

\begin{proof}
If \(k_{N+1}=k_N+1\), then \(M_{N+1}=M_N\); if
\(k_{N+1}=k_N-1\), then \(M_{N+1}=M_N+2\).
For an irrational endpoint, Proposition~\ref{prop:cf-coding} gives infinitely
many returns, and at the return times \(M_{\tau_r}=\tau_r\to\infty\).
\end{proof}

The resulting alternating constant and increasing blocks are shown in
Figure~\ref{fig:cf-blocks}.

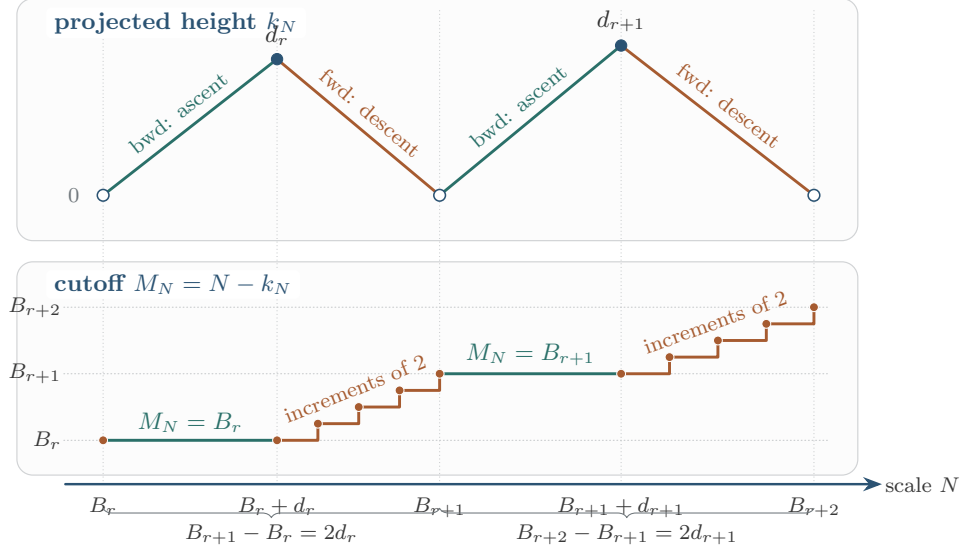
\begin{figure}[H]
\centering
\begin{tikzpicture}[x=1cm,y=1cm,>=Stealth]
  \begin{scope}[on background layer]
    \draw[panel] (-.34,3.52) rectangle (10.78,6.72);
    \draw[panel] (-.34,.42) rectangle (10.78,3.24);
    \foreach \x in {.80,3.10,5.25,7.65,10.20}{
      \draw[guide] (\x,.34)--(\x,6.58);
    }
  \end{scope}

  \node[figuretitle,anchor=west]
    at (.04,6.40) {projected height $k_N$};
  \draw[line width=1.08pt,draw=figteal]
    (.80,4.12)--(3.10,5.92)
    node[midway,sloped,above=3pt,smalllab,text=figteal] {$\bwd$: ascent};
  \draw[line width=1.08pt,draw=figorange]
    (3.10,5.92)--(5.25,4.12)
    node[midway,sloped,above=3pt,smalllab,text=figorange] {$\fwd$: descent};
  \draw[line width=1.08pt,draw=figteal]
    (5.25,4.12)--(7.65,6.10)
    node[midway,sloped,above=3pt,smalllab,text=figteal] {$\bwd$: ascent};
  \draw[line width=1.08pt,draw=figorange]
    (7.65,6.10)--(10.20,4.12)
    node[midway,sloped,above=3pt,smalllab,text=figorange] {$\fwd$: descent};
  \foreach \x/\y in {.80/4.12,5.25/4.12,10.20/4.12}{
    \node[openvertex,inner sep=1.55pt] at (\x,\y) {};
  }
  \node[vertex] at (3.10,5.92) {};
  \node[vertex] at (7.65,6.10) {};
  \node[lab,above=4pt] at (3.10,5.92) {$d_r$};
  \node[lab,above=4pt] at (7.65,6.10) {$d_{r+1}$};
  \node[smallgraylab,left] at (.52,4.12) {$0$};

  \node[figuretitle,anchor=west]
    at (.04,2.94) {cutoff $M_N=N-k_N$};
  \draw[guide] (.28,.88)--(10.40,.88);
  \draw[guide] (.28,1.76)--(10.40,1.76);
  \draw[guide] (.28,2.64)--(10.40,2.64);
  \node[smalllab,left] at (.28,.88) {$B_r$};
  \node[smalllab,left] at (.28,1.76) {$B_{r+1}$};
  \node[smalllab,left] at (.28,2.64) {$B_{r+2}$};

  \draw[line width=1.10pt,draw=figteal] (.80,.88)--(3.10,.88);
  \draw[line width=1.10pt,draw=figorange]
    (3.10,.88)--(3.64,.88)--(3.64,1.10)--(4.18,1.10)--
    (4.18,1.32)--(4.72,1.32)--(4.72,1.54)--(5.25,1.54)--(5.25,1.76);
  \draw[line width=1.10pt,draw=figteal] (5.25,1.76)--(7.65,1.76);
  \draw[line width=1.10pt,draw=figorange]
    (7.65,1.76)--(8.29,1.76)--(8.29,1.98)--(8.93,1.98)--
    (8.93,2.20)--(9.57,2.20)--(9.57,2.42)--(10.20,2.42)--(10.20,2.64);

  \foreach \x/\y in {
    .80/.88,3.10/.88,3.64/1.10,4.18/1.32,4.72/1.54,
    5.25/1.76,7.65/1.76,8.29/1.98,8.93/2.20,9.57/2.42,10.20/2.64}{
    \node[circle,draw=white,fill=figorange,line width=.30pt,inner sep=1.18pt]
      at (\x,\y) {};
  }
  \node[lab,text=figteal,above=2pt] at (1.95,.88)
    {$M_N=B_r$};
  \node[smalllab,text=figorange,rotate=22,above=3pt] at (4.18,1.34)
    {increments of $2$};
  \node[lab,text=figteal,above=2pt] at (6.45,1.76)
    {$M_N=B_{r+1}$};
  \node[smalllab,text=figorange,rotate=20,above=3pt] at (8.94,2.22)
    {increments of $2$};

  \draw[->,axis] (.30,.30)--(11.10,.30) node[right,smalllab] {scale $N$};
  \node[smalllab,below=4pt] at (.80,.30) {$B_r$};
  \node[smalllab,below=4pt] at (3.10,.30) {$B_r+d_r$};
  \node[smalllab,below=4pt] at (5.25,.30) {$B_{r+1}$};
  \node[smalllab,below=4pt] at (7.65,.30) {$B_{r+1}+d_{r+1}$};
  \node[smalllab,below=4pt] at (10.20,.30) {$B_{r+2}$};
  \draw[decorate,decoration={brace,mirror,amplitude=3.5pt},draw=figgray]
    (.80,-.02)--(5.25,-.02)
    node[midway,below=5pt,smalllab,text=black!78] {$B_{r+1}-B_r=2d_r$};
  \draw[decorate,decoration={brace,mirror,amplitude=3.5pt},draw=figgray]
    (5.25,-.02)--(10.20,-.02)
    node[midway,below=5pt,smalllab,text=black!78]
    {$B_{r+2}-B_{r+1}=2d_{r+1}$};
\end{tikzpicture}
\caption{Projected height excursions (top) and exact cutoff blocks (bottom).
On a backward/ascent half, \(M_N\) is constant; on a forward/descent half, it
increases by \(2\) at each scale.  The value \(B_r+2d_r=B_{r+1}\) belongs to
the next backward block, so there is no extra \(+2\) in the return time.}
\label{fig:cf-blocks}
\end{figure}

\subsection{Bounded observables and finite-window exactness}
\begin{theorem}[Moving-shadow bound]\label{thm:dense-bounded}
Let \(x\in X\).  For every \(N\ge0\) such that \(M_N\ge0\), and every bounded
right \(K\)-invariant observable \(f=\Phi\circ h_K\),
\[
\left|\lambda_{x,N}^{\mathrm{den}}(f)-\int_Xf\,d\mu_X\right|
 \ll_q\|\Phi\|_\infty q^{-M_N/2}.
\]
If \(\xi_g\) is irrational, then
\[
\lambda_{x,N}^{\mathrm{den}}(f)\longrightarrow\int_Xf\,d\mu_X.
\]
\end{theorem}

\begin{proof}
Combine the exact identity \eqref{eq:dense-exact-shadow} with
Proposition~\ref{prop:moving-TV}.  The convergence follows from
Corollary~\ref{cor:MN-diverges}.
\end{proof}

\begin{corollary}[Eventual exactness for finite-height tests]
\label{cor:dense-finite-exact}
Suppose \(\xi_g\) is irrational.  Let \(R\in2\ZZ_{\ge0}\), and let
\(\Phi:2\ZZ_{\ge0}\to\CC\) satisfy
\(\operatorname{supp}\Phi\subset\{0,2,\ldots,R\}\).  Then, for all sufficiently
large \(N\),
\[
\lambda_{x,N}^{\mathrm{den}}(\Phi\circ h_K)
 =\int_X\Phi\circ h_K\,d\mu_X.
\]
\end{corollary}

\begin{proof}
By Corollary~\ref{cor:MN-diverges}, eventually \(M_N>R\).  Apply
Proposition~\ref{prop:finite-window-exact} to
\eqref{eq:dense-exact-shadow}.
\end{proof}

\begin{remark}
After adjoining the compact factor \(M_\infty\), the sets \(M_\infty U_N\)
(considered separately in the two parity classes) are the full-horospherical
objects naturally compared with the framework of \cite{CFS2022}.  We do not
invoke its equidistribution theorem here.  On right \(K\)-invariant observables
the compact factor does not change the average, as explained in
Section~\ref{sec:spherical-dictionary}.
The additional information here is the exact finite-scale identity
\eqref{eq:dense-exact-shadow}, the explicit discrepancy, and the finite-window
equality.
\end{remark}

\subsection{Unbounded profiles}
For a moving forward root, \(L^1(\rho^{\ev})\) alone does not control the
always-forward atom and all late first turns.  The following condition isolates
the two additional terms.

\begin{definition}[Moving-root admissibility]
\label{def:moving-admissible}
Fix \(x=\Gamma g\in X\), and retain the notation
\(k_N=k_N(x)\), \(\varepsilon_N=\varepsilon_N(x)\), and \(M_N=M_N(x)\) from
the canonical truncation above.  Assume \(M_N\to\infty\).  A height profile
\(\Phi:2\ZZ_{\ge0}\to\CC\) is \emph{admissible for the moving root attached to
\(x\)} if
\[
\Phi\in L^1(\rho^{\ev}),\qquad
q^{-N}|\Phi(k_N+N)|\longrightarrow0,
\]
and
\begin{equation}\label{eq:window-admissibility}
\mathcal W_N(\Phi;x)
 :=q^{-M_N/2}
 \sum_{\substack{0\le d\le N+k_N-2\\ d\equiv M_N\ (2)}}
 q^{-d/2}|\Phi(d)|\longrightarrow0.
\end{equation}
Here \(d\) is the height of the point mass produced by a late first turn; up
to a constant depending only on \(q\), the factor
\(q^{-M_N/2}q^{-d/2}\) is the corresponding first-turn weight.
An empty sum is interpreted as zero.  The window condition is automatic for
bounded \(\Phi\), and also for \(|\Phi(2m)|\ll q^{\alpha m}\) with \(\alpha<1\).
\end{definition}

\begin{theorem}[Admissible \(L^1\)-profiles]
\label{thm:dense-L1}
Fix \(x=\Gamma g\in X\), retain the notation \(k_N,\varepsilon_N,M_N\) above,
and assume \(M_N\to\infty\).  If \(\Phi\) is admissible for the moving root
attached to \(x\), then
\[
\lambda_{x,N}^{\mathrm{den}}(\Phi\circ h_K)
 \longrightarrow
 \sum_{m\ge0}\rho^{\ev}(2m)\Phi(2m).
\]
\end{theorem}

\begin{proof}
Write \(k=k_N\) and \(M=M_N=N-k\).  It is enough to consider the eventually
nonnegative values of \(M\).  In a backward state,
Proposition~\ref{prop:back-discrepancy} gives
\[
\int \Phi\,d\mu^{\bwd}_{k,N}-\ang\Phi\rang_{\rho^{\ev}}
 =q^{-M-1}\Phi(M)
  -\sum_{2m\ge M+2}\rho^{\ev}(2m)\Phi(2m).
\]
Both terms tend to zero because \(\Phi\in L^1(\rho^{\ev})\) and
\(\rho^{\ev}(M)\asymp_q q^{-M}\).

In a forward state, Proposition~\ref{prop:forward-discrepancy} separates the
always-forward atom from the first-turn sum.  The atom tends to zero by the
second condition in Definition~\ref{def:moving-admissible}; its constant part
is \(q^{-N}\ang\Phi\rang_{\rho^{\ev}}=o(1)\).  For a first turn
at time \(r+1\), put
\[
D_r=N-r-1,\qquad \ell_r=k+r-1,\qquad
M_r=D_r-\ell_r=M-2r.
\]
For \(r\le M/2\), Proposition~\ref{prop:back-discrepancy} bounds the
nonconstant contribution by
\[
\sum_{r\le M/2}q^{-r}
 \sum_{2m\ge M-2r}\rho^{\ev}(2m)|\Phi(2m)|.
\]
Given \(R\), the part \(r\le R\) tends to zero as a finite sum of
\(L^1\)-tails, and the part \(r>R\) is at most a constant times
\(\|\Phi\|_{L^1(\rho^{\ev})}\sum_{r>R}q^{-r}\).  Hence the early-turn contribution
vanishes.

For \(r>M/2\), Lemma~\ref{lem:back-before-bottom} makes the backward subshadow
a point mass at \(d=2r-M\).  Therefore the late-turn contribution involving
\(\Phi\) is bounded, up to a constant depending only on \(q\), by
\[
q^{-M/2}
 \sum_{\substack{0\le d\le N+k-2\\d\equiv M\ (2)}}
 q^{-d/2}|\Phi(d)|
 =\mathcal W_N(\Phi;x),
\]
which tends to zero.  The constant part of the late-turn discrepancy is
\(O(q^{-M/2})\).  This proves the theorem.
\end{proof}

\subsection{Growth classes and continued-fraction rates}
\begin{theorem}[Controlled growth for moving roots]
\label{thm:dense-growth}
Fix \(x=\Gamma g\in X\) and retain the notation
\(k_N=k_N(x)\), \(\varepsilon_N=\varepsilon_N(x)\), and \(M_N=M_N(x)\).
Let \(\Phi:2\ZZ_{\ge0}\to\CC\), and assume that, for some \(C>0\),
\[
|\Phi(2m)|\le Cq^{\alpha m}\qquad(m\ge0)
\]
with \(\alpha<2\), and put \(\beta=(2-\alpha)/2\).  Whenever \(M_N\ge0\), the
following estimates hold.
\begin{enumerate}[label=(\roman*)]
\item If \(\alpha<1\), then
\[
\left|\lambda_{x,N}^{\mathrm{den}}(\Phi\circ h_K)
       -\ang\Phi\rang_{\rho^{\ev}}\right|
 \ll_{q,\alpha}Cq^{-M_N/2}.
\]
\item If \(\alpha=1\), then
\[
\left|\lambda_{x,N}^{\mathrm{den}}(\Phi\circ h_K)
       -\ang\Phi\rang_{\rho^{\ev}}\right|
 \ll_q C\,(M_N+k_N+1)q^{-M_N/2}.
\]
\item If \(1<\alpha<2\), then
\[
\left|\lambda_{x,N}^{\mathrm{den}}(\Phi\circ h_K)
       -\ang\Phi\rang_{\rho^{\ev}}\right|
 \ll_{q,\alpha}Cq^{-\beta M_N+(\alpha-1)k_N}.
\]
\end{enumerate}
Consequently, convergence holds in (i) if \(M_N\to\infty\), in (ii) if
\[
(M_N+k_N+1)q^{-M_N/2}\to0,
\]
and in (iii) if
\[
-\frac{2-\alpha}{2}M_N+(\alpha-1)k_N\to-\infty.
\]
For \(1<\alpha<2\), a sufficient condition is that, for some
\(\vartheta\ge0\),
\[
k_N\le\vartheta M_N\quad\hbox{eventually},\qquad
\vartheta<\frac{2-\alpha}{2(\alpha-1)}.
\]
\end{theorem}

\begin{proof}
The growth hypothesis and \(\alpha<2\) imply
\(\Phi\in L^1(\rho^{\ev})\) and
\(|\ang\Phi\rang_{\rho^{\ev}}|\ll_{q,\alpha}C\).
For a backward state, Proposition~\ref{prop:back-discrepancy} gives
\[
\left|\int \Phi\,d\mu^{\bwd}_{k_N,N}
       -\ang\Phi\rang_{\rho^{\ev}}\right|
 \ll_{q,\alpha}Cq^{-\beta M_N}.
\]
This is bounded by the respective right-hand sides above.

Suppose the state is forward, and write \(k=k_N\), \(M=M_N=N-k\).  The
always-forward atom satisfies
\[
q^{-N}|\Phi(k+N)|
 \le Cq^{-N}q^{\alpha(k+N)/2}
 =Cq^{-\beta M+(\alpha-1)k}.
\]
The omitted constant part of this signed term is
\(q^{-N}|\ang\Phi\rang_{\rho^{\ev}}|\ll_{q,\alpha}Cq^{-N}\), which is smaller
than each asserted bound.
For early first turns, the bound from
Proposition~\ref{prop:effective-back} is
\[
Cq^{-\beta M}
 \sum_{r\le M/2}q^{(1-\alpha)r}.
\]
This is \(O(Cq^{-M/2})\) for \(\alpha<1\),
\(O(C(M+1)q^{-M/2})\) for \(\alpha=1\), and
\(O(Cq^{-\beta M})\) for \(1<\alpha<2\).

For late turns, set \(d=2r-M\).  The contribution containing \(\Phi\) is at
most
\[
Cq^{-M/2}
 \sum_{\substack{0\le d\le N+k-2\\d\equiv M\ (2)}}
q^{(\alpha-1)d/2}.
\]
The contribution containing \(\ang\Phi\rang_{\rho^{\ev}}\) is
\(O_{q,\alpha}(Cq^{-M/2})\), and is absorbed by the same three bounds.
The finite geometric sum is bounded if \(\alpha<1\), has
\(O(M+k+1)\) terms if \(\alpha=1\), and, if \(1<\alpha<2\), is bounded by a
constant times
\[
q^{(\alpha-1)(M+2k)/2}.
\]
The latter gives \(O(Cq^{-\beta M+(\alpha-1)k})\).  Combining these estimates
proves the theorem.
\end{proof}

\begin{corollary}[Rates in continued-fraction coordinates]
\label{cor:cf-rates}
Assume that \(\xi_g\) is irrational and index the return times and digits as in
Proposition~\ref{prop:cf-coding}.  Under the growth hypothesis of
Theorem~\ref{thm:dense-growth}, set
\[
E_N(\Phi;x)=
\left|\lambda_{x,N}^{\mathrm{den}}(\Phi\circ h_K)
      -\ang\Phi\rang_{\rho^{\ev}}\right|,
\qquad \beta=\frac{2-\alpha}{2}.
\]
On the ascending half, where \(N=B_r+j\) and \(0\le j\le d_r-1\), one has
\[
E_{B_r+j}(\Phi;x)\ll_{q,\alpha}Cq^{-\beta B_r}.
\]
On the descending half, where \(N=B_r+d_r+j\) and
\(0\le j\le d_r-1\), one has
\[
E_{B_r+d_r+j}(\Phi;x)\ll
\begin{cases}
Cq^{-B_r/2-j},&\alpha<1,\\
C(B_r+d_r+j+1)q^{-B_r/2-j},&\alpha=1,\\
Cq^{-\beta B_r+(\alpha-1)d_r-j},&1<\alpha<2,
\end{cases}
\]
where the implicit constant depends only on \(q,\alpha\) (only on \(q\) in
the middle line).  Thus the factor at \(\alpha=1\) is linear in the scale
\(N=B_r+d_r+j\), equivalently logarithmic in the ball volume \(q^N\).
\end{corollary}

\begin{proof}
During ascent, Proposition~\ref{prop:cf-coding} gives the backward state with
\(M_N=B_r\), so the backward estimate in the proof of
Theorem~\ref{thm:dense-growth} gives the first bound.  During descent it gives
the forward state with \(M_N=B_r+2j\) and \(k_N=d_r-j\).  Substitution into
the three forward bounds of that theorem gives the displayed cases; in the
last one,
\[
-\beta(B_r+2j)+(\alpha-1)(d_r-j)
=-\beta B_r+(\alpha-1)d_r-j.
\]
\end{proof}

\section{Conclusion}\label{sec:conclusion}
The \(K\)-spherical restriction turns the Nagao quotient into a nearest-neighbor
recursion on the oriented state space
\(\{\fwd,\bwd\}\times\ZZ_{\ge0}\).  Its asymmetry---one edge toward the cusp
and \(q\) edges toward the compact part---is precisely what produces the exact
top-shell correction, the missing tail, and the first-turn mixture.

For compact \(U\)-orbits, the right-action convention and the normal form in
Proposition~\ref{prop:compact-normal-form} lead to the single-shadow identity
\eqref{eq:compact-exact-shadow}.  For arbitrary orbits, the standard balls
\(t^N\cO\) lead to the moving-shadow identity
\eqref{eq:dense-exact-shadow}.  The continued-fraction digits do not replace
the geometric argument; they give an arithmetic description of the successive
excursion depths and hence of the scale-dependent cutoff.

The conclusions of this paper concern right \(K\)-invariant observables.
The broader horospherical literature addresses non-spherical questions under
its own hypotheses, but no non-spherical equidistribution statement for the
specific \(U_N\)-averages is deduced here.  What is special to the present model
is that the spherical error can be computed at every finite scale and vanishes
identically on every fixed finite height window once the cutoff has passed that
window.

\end{document}